\documentclass{CUP-JNL-BLG}

\usepackage[dvipsnames]{xcolor}
\usepackage{latexsym}
\usepackage{graphicx}
\usepackage{multicol,multirow}
\usepackage{amsmath,amssymb,amsfonts}
\numberwithin{equation}{section}
\usepackage{soul}
\usepackage{blkarray}
\usepackage{hhline}
\usepackage{mathrsfs}
\usepackage{amsthm}
\usepackage{rotating}
\usepackage{longtable}
\usepackage{diagbox}
\usepackage{array}
\usepackage{makecell}
\usepackage[overload]{empheq}
\usepackage{appendix}
\usepackage[numbers,super,sort&compress,round]{natbib}
\usepackage{ifpdf}
\usepackage[T1]{fontenc}
\usepackage{times}
\usepackage{newtxmath}
\usepackage{textcomp}%
\usepackage{hyperref}
\usepackage{lipsum}
\usepackage{cleveref}
\usepackage{pifont}% http://ctan.org/pkg/pifont
\DeclareMathOperator*{\argmin}{arg\,min}

\newtheorem{theorem}{Theorem}[section]

\newtheorem{proposition}[theorem]{Proposition}

\theoremstyle{definition}
\newtheorem{definition}[theorem]{Definition}

\newtheorem{remark}[theorem]{Remark}

\newtheorem{example}{Example}

\graphicspath{{./figures/}}

\usepackage{subcaption}

\usepackage{algorithm}
\usepackage{algpseudocode}
\usepackage{myalgorithm}

\usepackage{graphicx}
\usepackage[Export]{adjustbox}
\usepackage{tabularx}
\usepackage{cellspace}

\begin{document}

\begin{Frontmatter}

\title[Article Title]{Algebraic Constraints and Algorithms for Common Lines in Cryo-EM}

\author*[1]{Tommi Muller}\email{tommi.muller@maths.ox.ac.uk}\orcid{0000-0003-4333-920X}
\author[2]{Adriana L. Duncan}\email{aduncan@math.utexas.edu}\orcid{0000-0001-5075-858X}
\author[3]{Eric J. Verbeke}\email{ev9102@princeton.edu}\orcid{https://orcid.org/0000-0001-9506-424X}
\author[2,4]{Joe Kileel}\email{jkileel@math.utexas.edu}\orcid{0000-0001-9926-9170}

\address[1]{\orgdiv{Mathematical Institute}, \orgname{University of Oxford}, \orgaddress{\city{Oxford}, \postcode{OX2 6GG},  \country{United Kingdom}}}

\address[2]{\orgdiv{Department of Mathematics}, \orgname{University of Texas at Austin}, \orgaddress{\city{Austin}, \postcode{78712}, \state{TX},  \country{United States}}}

\address[3]{\orgdiv{Program in Applied and Computational Mathematics}, \orgname{Princeton University}, \orgaddress{\city{Princeton}, \postcode{08540}, \state{NJ},  \country{United States}}}

\address[4]{\orgdiv{Oden Institute}, \orgname{University of Texas at Austin}, \orgaddress{\city{Austin}, \postcode{78712}, \state{TX},  \country{United States}}}


\abstract{
We revisit the topic of common lines between projection images in single particle cryo-electron microscopy (cryo-EM). We derive a novel low-rank constraint on a certain $2n \times n$ matrix storing properly-scaled basis vectors for the common lines between $n$ projection images of one molecular conformation.  
Using this algebraic constraint and others, we give optimization algorithms to denoise common lines and recover the unknown 3D rotations associated to the images. 
As an application, we develop a clustering algorithm to partition a set of noisy images into homogeneous communities using common lines, in the case of discrete heterogeneity in cryo-EM.
We demonstrate the methods on synthetic and experimental datasets. 
}

\begin{policy}[Impact Statement]
Single particle cryo-electron microscopy is an imaging technique used to determine the 3D structure of biomolecules from noisy 2D projection images.  This paper revisits one of the first approaches to cryo-EM image processing, namely common lines between pairs of 2D class averages coming from the Fourier slice theorem.  
We present a novel mathematical approach for dealing with common lines: in contrast to some alternatives, it operates directly on the common lines themselves and avoids triplewise angular reconstitution completely. The paper then derives novel algebraic constraints on sets of consistent common lines, including a straightforward low-rank matrix condition.
The algebraic conditions are incorporated into optimization methods arising from the field of computer vision to produce new methods for computational tasks involving common lines.
In particular, we achieve improved accuracy in common line denoising and rotation recovery at low signal-to-noise ratios.
We also present a method to detect homogeneous communities of 2D class averages in the case of a cryo-EM dataset with multiple molecular conformations.
Altogether this work clarifies a classic topic in cryo-EM, and opens the door to applying common lines techniques on more challenging \nolinebreak datasets.

\end{policy}

\end{Frontmatter}

\section{Introduction}

Single particle cryo-electron microscopy (cryo-EM) is an imaging technique capable of recovering the high resolution 3D structure of molecules from many noisy tomographic projection images taken at unknown viewing angles. 
One of the first approaches for 3D reconstruction, known as angular reconstitution, is based on the \text{common line property} of projection images induced by the Fourier slice theorem~\cite{Vainshtein_determination_1986, VANHEEL1987111}. Due to the low signal-to-noise ratio (SNR) in cryo-EM data, detecting common lines is a difficult task~\cite{SINGER2010312}: even today when applied to 
denoised averages of images, referred to as 2D class averages.
Detecting common lines is subject to angular errors and incorrectly identified common lines. 
Although methods which seek to minimize global errors in the estimated viewing directions have increased the utility of common lines methods~\cite{shkolniskySinger2}, additional constraints on common lines are needed to improve their accuracy and robustness.

In this paper we propose a novel approach for dealing with common lines.  Specifically, we assemble the  estimated common lines for a dataset of $n$ images into a certain $2n \times n$ matrix, which stores properly-scaled basis vectors for the common lines (\Cref{prop:rank}). 
The matrix directly encodes common lines data, without requiring angular reconstitution on various subsets of images or needing voting procedures like some existing formulations \cite{shkolniskySinger1,SINGER2010312}.  As such, it yields a direct and more global approach than prior constructions for  common \nolinebreak lines.

As a main contribution, we derive \text{algebraic constraints} on the matrix of common lines, which must be satisfied in order for a set of common lines to be consistent with a single asymmetric molecular conformation.  The constraints include a straightforward low-rank condition on the matrix, as well as various sparse quadratic constraints.  
Importantly, the constraints enable new strategies for \text{computational tasks} involving common lines, in particular for: denoising common lines; estimating 3D rotations; and clustering heterogeneous image sets into homogeneous subsets.
We demonstrate this by adapting optimization algorithms from other domains to these tasks, using the algebraic constraints.
We remark that our constraints seem better suited for numerical optimization than the semialgebraic constraints found in prior work~\cite{dynerman2014semi}.

Notably, the clustering problem is a recent application of common lines \cite{verbeke_separating_2020}.  In more detail, the goal is to sort discretely heterogeneous image sets of multiple molecules into communities corresponding to homogeneous image subsets.
This application is motivated by the increasing complexity of cryo-EM datasets, where samples may not be purified and thus the number of distinct molecules contained in a dataset is more than one ~\cite{verbeke_classification_2018, yi_electron_2019, sae-lee_protein_2022}. 
Our algebraic constraints and optimization algorithms enable consistency checks of subsets of images, to test whether the subset corresponds to a single molecule.

As a mathematical guarantee, we prove that computing the correct scales in the homogeneous case admits an essentially unique global optimum, see \Cref{thm:row-column-scales}. 
We implement our algorithms and test them on simulated and real datasets in \Cref{sec:experiments}.
%\eric{we implement our algorithm and test on..}. 
The results demonstrate that our methods can be successful when applied to 2D class averages at noise levels comparable to experimental data, in both the homogeneous and discretely heterogeneous cases.  
We conclude with a discussion of potential future improvements.

\subsection*{Advantages.}
There are several advantages to our approach for dealing with common lines:
\begin{itemize}
\item The new formulation is directly in terms of the data, that is, in terms of the common lines themselves.
\item It involves multiple common lines simultaneously, and does not require triplewise angular reconstitution at all (in contrast to \cite{shkolniskySinger1} for instance), making our approach fully global and potentially more robust to noise than alternatives.
\item The algebraic constraints can be incorporated into  existing optimization algorithms that have seen success in computer vision applications \cite{scalingAmit}. 
\item The resulting algorithms outperform existing methods for denoising and rotation recovery on noisy simulated data, and perform comparably well for clustering heterogeneous image sets on real data, even though the optimization algorithms are off-the-shelf.
\end{itemize}

\section{Background}

First, we recall a standard simplified mathematical model for cryo-EM, in the homogeneous case of one molecular conformation.   
We assume there exists a 3D function $\varphi : \mathbb{R}^3 \rightarrow \mathbb{R}$ describing the electrostatic potential generated by the molecule. 
As data, we receive $n$ two-dimensional tomographic projection images, denoted $I_{R^{(i)}} : \mathbb{R}^2 \rightarrow \mathbb{R}$ for $i=1, \ldots, n$, where $R^{(i)} \in \operatorname{SO}(3)$ are 3D rotations associated with each image. 
The goal of single particle cryo-EM is to recover the underlying 3D structure $\varphi$ from the set of 2D tomographic projection images which are observed at unknown rotations. 
The images, in their idealized and noiseless form, have the following Fourier transforms due to the Fourier slice theorem:
\begin{equation}\label{eq:Ihat-first}
\widehat{I}_{R^{(i)}}(\hat{x}, \hat{y}) = (R^{(i)} \cdot \widehat{\varphi})(\hat{x}, \hat{y}, 0).
\end{equation}
Here $\widehat{\varphi}(\hat{x}, \hat{y}, \hat{z}) := \int_{\mathbb{R}^3 }\varphi(x,y,z) e^{\sqrt{-1} (x \hat{x} + y \hat{y} + z \hat{z})} dx dy dz$ denotes the Fourier transform of 
$\varphi$, and 
$R^{(i)} \cdot \widehat{\varphi}$ denotes the rotation of $\widehat{\varphi}$ by $R^{(i)}$.
Writing $R^{(i)} = \begin{pmatrix}  {\mathbf{r}_1^{(i)}} &   {\mathbf{r}_2^{(i)}} & {\mathbf{r}_3^{(i)}} \end{pmatrix}^{\!\top}$, equation~\eqref{eq:Ihat-first} reads
\begin{equation}\label{eq:I-hat}
\widehat{I}_{R^{(i)}}(\hat{x}, \hat{y}) \,\, = \,\, \widehat{\varphi}\left((R^{(i)})^{\top} ( \hat{x} \,\, \hat{y} \,\, 0 )^{\top}\right) \,\, = \,\, \widehat{\varphi}(\hat{x}\mathbf{r}_1^{(i)} + \hat{y} \mathbf{r}_2^{(i)}).
\end{equation}
Generically, for asymmetric molecules $\varphi$ and distinct rotations $R^{(i)}$ and $R^{(j)}$, there exist unique lines through the origin in the domain of the Fourier-transformed images $\widehat{I}_{R^{(i)}}$ and $\widehat{I}_{R^{(j)}}$, respectively $\ell_{ij} \subseteq \operatorname{domain}(\widehat{I}_{R^{(i)}}) = \mathbb{R}^2$ and $\ell_{ji} \subseteq \operatorname{domain}(\widehat{I}_{R^{(j)}}) = \mathbb{R}^2$, such that the restrictions
\begin{equation}\label{eq:restriction}
\widehat{I}_{R^{(i)}} \big{|}_{\ell_{ij}} \,\, = \,\, \widehat{I}_{R^{(j)}} \big{|}_{\ell_{ji}} 
\end{equation}
are equal as functions on $\mathbb{R}^2$. In cryo-EM, one says that $\ell_{ij}$ and $\ell_{ji}$ are the \textbf{\textit{common lines}} between the $i$th and $j$th image. 
In modest-noise settings, which are arrived at by working with 2D class averages instead of raw tomographic images \cite{frank2006three}, common lines can be estimated from real cryo-EM data. 
They give basic ways to do 3D reconstruction in cryo-EM;  for example, see the angular reconstitution technique of van Heel~\cite{VANHEEL1987111} or the works of Shkolnisky, Singer and their collaborators ~\cite{SINGER2010312, shkolniskySinger1, shkolniskySinger2} for example.

From equation~\eqref{eq:I-hat}, the common lines $\ell_{ij}$ and $\ell_{ji}$ may be found mathematically by expressing the single line in 3D space:
\begin{align} \label{eq:intersection}
\operatorname{span}(\mathbf{r}_1^{(i)}, \mathbf{r}_2^{(i)}) \cap \operatorname{span}(\mathbf{r}_1^{(j)}, \mathbf{r}_2^{(j)}) &= \operatorname{span}(\mathbf{r}_3^{(i)})^{\perp} \cap  \operatorname{span}(\mathbf{r}_3^{(j)})^{\perp}  = \operatorname{span}(\mathbf{r}_3^{(i)} \times \mathbf{r}_3^{(j)}) \subseteq \operatorname{dom}(\widehat{\varphi}) = \mathbb{R}^3
\end{align}
\medskip
% \medskip
\noindent in the coordinate system of the $i$th and $j$th image respectively.  Here, $\times$ denotes the cross product in $\mathbb{R}^3$. 
Combining~\eqref{eq:I-hat} and~\eqref{eq:intersection}, equation~\eqref{eq:restriction} may be written as:
\begin{align*}
& \widehat{I}_{R^{(i)}}(\lambda \hat{x}_{ij} , \lambda \hat{y}_{ij}) \,\, = \,\, \widehat{I}_{R^{(j)}}(\lambda \hat{x}_{ji} , \lambda \hat{y}_{ji}) \quad \text{for all } \lambda \in \mathbb{R}, \nonumber \\[0.5em]
& \text{where } \, \hat{x}_{ij} := \langle \mathbf{r}_1^{(i)}, \mathbf{r}_3^{(i)} \times \mathbf{r}_3^{(j)} \rangle, \,\, \hat{y}_{ij} := \langle \mathbf{r}_2^{(i)}, \mathbf{r}_3^{(i)} \times \mathbf{r}_3^{(j)} \rangle, \\[0.5em]
& \text{and } \, \hat{x}_{ji} := -\langle \mathbf{r}_1^{(j)}, \mathbf{r}_3^{(i)} \times \mathbf{r}_3^{(j)} \rangle, \,\, \hat{y}_{ji} := -\langle \mathbf{r}_2^{(j)}, \mathbf{r}_3^{(i)} \times \mathbf{r}_3^{(j)} \rangle \nonumber
\end{align*}
where $\langle \cdot , \cdot \rangle$ denotes the standard inner product in $\mathbb{R}^3$. Common lines can therefore be encoded via:

\medskip

\begin{definition} \label{def:representatives}
Vectors $\mathbf{a}_{ij}, \mathbf{a}_{ji} \in \mathbb{R}^2$ are called \textbf{\textit{representatives}} for the common lines $\ell_{ij}$ and $\ell_{ji}$ if there exists a nonzero scalar $\lambda_{ij} = \lambda_{ji} \in \mathbb{R}$ such that 
\begin{align}{\label{eq:representatives}}
    \mathbf{a}_{ij} = \lambda_{ij}\begin{pmatrix} \langle \mathbf{r}_1^{(i)}, \mathbf{r}_3^{(i)} \times \mathbf{r}_3^{(j)} \rangle \\ \langle \mathbf{r}_2^{(i)}, \mathbf{r}_3^{(i)} \times \mathbf{r}_3^{(j)} \rangle 
    \end{pmatrix}, &\quad \mathbf{a}_{ji} = \lambda_{ji}\begin{pmatrix} -\langle \mathbf{r}_1^{(j)}, \mathbf{r}_3^{(i)} \times \mathbf{r}_3^{(j)} \rangle \\ -\langle \mathbf{r}_2^{(j)}, \mathbf{r}_3^{(i)} \times \mathbf{r}_3^{(j)} \rangle
    \end{pmatrix}.
\end{align}
\end{definition}
Equivalently, representatives $\mathbf{a}_{ij}$ and $\mathbf{a}_{ji}$ are choices of basis vectors for the common lines $\ell_{ij}$ and $\ell_{ji}$ which satisfy $\widehat{I}_{R^{(i)}}(\lambda \mathbf{a}_{ij}) = \widehat{I}_{R^{(j)}}(\lambda \mathbf{a}_{ji})$ for all $\lambda \in \mathbb{R}$.
Representatives for common lines can be estimated from 2D class averages in practice. 

\medskip

We stress that, although quite standard, the model \eqref{eq:Ihat-first} is greatly simplified.  
It neglects the effects of contrast transfer functions (CTFs), imperfect centering in particle picking, and blurring in class averaging.  
Further, we have restricted attention to the case of asymmetric molecules, as otherwise common lines are only unique up to the action of the relevant symmetry group (e.g., see \cite{geva2023common}).

\section{Constraints on sets of common lines}

\subsection{The common lines matrix}\label{sec:common-matrix}

We introduce an object to keep track of all common lines in a dataset. 
It is the main object in this paper.

\medskip

\begin{definition} \label{def:common-lines-matr}
A \textit{\textbf{common lines matrix}} associated to rotations $R^{(1)}, \ldots, R^{(n)} \in \operatorname{SO}(3)$ is a matrix $A \in \mathbb{R}^{2n \times n}$, which when regarded as an $n \times n$ block matrices with $2 \times 1$ blocks $\mathbf{a}_{ij} \in \mathbb{R}^2$ is such that $\mathbf{a}_{ij}, \mathbf{a}_{ji}$ are representatives for the common lines $\ell_{ij}, \ell_{ji}$ if $i$ and $j$ are distinct and $\mathbf{a}_{ii} = 0$ otherwise. 
If the scalars $\lambda_{ij}$ in~\eqref{eq:representatives} are all equal to 1, then we call $A$ the \textit{\textbf{pure common lines matrix}}.
\end{definition}

\medskip

Thus a common lines matrix $A$ associated to $R^{(1)}, \ldots, R^{(n)}$ is uniquely defined up to $\binom{n}{2}$ nonzero real scalars $\lambda_{ij}$ ($i<j$).  
In real data settings where the 2D class averages are sufficiently denoised, we can estimate $A$ from data by estimating representatives for the common lines. 

We now present constraints which a \textit{pure} common lines matrix must satisfy.  Firstly, there is the following low-rank condition.  All of our computational methods take advantage of this.

\medskip

\begin{theorem}{\label{prop:rank}}
Let $A \in \mathbb{R}^{2n \times n}$ be the pure common lines matrix associated to Zariski-generic rotations $R^{(1)}, \ldots, R^{(n)} \in \operatorname{SO}(3)$ where $n \geq 3$. Then $\operatorname{rank}(A) = 3$.
\end{theorem}

\begin{proof}
Since
\begin{align*}
\mathbf{a}_{ij} = \begin{pmatrix} 
\langle \mathbf{r}_1^{(i)}, \mathbf{r}_3^{(i)} \times \mathbf{r}_3^{(j)} \rangle \\ \langle \mathbf{r}_2^{(i)}, \mathbf{r}_3^{(i)} \times \mathbf{r}_3^{(j)} \rangle 
\end{pmatrix} =
\begin{pmatrix} 
\langle \mathbf{r}_1^{(i)} \times \mathbf{r}_3^{(i)}, \mathbf{r}_3^{(j)} \rangle \\ \langle \mathbf{r}_2^{(i)} \times \mathbf{r}_3^{(i)}, \mathbf{r}_3^{(j)} \rangle 
\end{pmatrix} = \begin{pmatrix} 
- \langle \mathbf{r}_2^{(i)}, \mathbf{r}_3^{(j)} \rangle \\ \langle \mathbf{r}_1^{(i)}, \mathbf{r}_3^{(j)} \rangle 
\end{pmatrix}
 =  \begin{pmatrix}  - {\mathbf{r}_2^{(i)}}^{\!\top}  \\  {\mathbf{r}_1^{(i)}}^{\!\top}  \end{pmatrix}_{2 \times 3} \mathbf{r}_3^{(j)},
\end{align*}
the pure commons line matrix admits the following factorization:
\begin{equation}{\label{eq:A-factorization}}
A =  \begin{pmatrix} - {\mathbf{r}_2^{(1)}}^{\!\top} \\ {\mathbf{r}_1^{(1)}}^{\!\top} \\ 
\vdots \\  -{\mathbf{r}_2^{(n)}}^{\!\top} \\ {\mathbf{r}_1^{(n)}}^{\!\top}  \end{pmatrix}_{2n \times 3} \!\!\! \begin{pmatrix}  \mathbf{r}_3^{(1)} & \cdots & \mathbf{r}_3^{(n)} \end{pmatrix}_{3 \times n}.
\end{equation}
Equation \eqref{eq:A-factorization} witnesses $\operatorname{rank}(A) \leq \min(3,n)$ $ = 3$.  We have equality when $R^{(i)}$ are generic because the two matrices in the factorization are full rank.
\end{proof}

There are also necessary quadratic constraints in the entries of a pure common lines matrix.

\medskip

\begin{proposition}{\label{prop:length}}
    Suppose $A \in \mathbb{R}^{2n \times n}$ is a pure common lines matrix where $n \geq 3$. Then for any $1 \leq i < j \leq n$, we have $\|\mathbf{a}_{ij}\|_2^2 = \|\mathbf{a}_{ji}\|_2^2$.
\end{proposition}

\begin{proof}
    See Appendix \ref{appendix:a}.
\end{proof}

\begin{proposition}{\label{prop:determinant}}
Suppose $A \in \mathbb{R}^{2n \times n}$ is a pure common lines matrix where $n \geq 3$. Then for any $1 \leq i < j < k \leq n$, we have $\mathrm{det}\begin{pmatrix}
    \mathbf{a}_{ij} & \mathbf{a}_{ik}
\end{pmatrix} = -\mathrm{det} \begin{pmatrix}
    \mathbf{a}_{ji} & \mathbf{a}_{jk}
\end{pmatrix} = \mathrm{det}\begin{pmatrix}
    \mathbf{a}_{ki} & \mathbf{a}_{kj}
\end{pmatrix}$.

\end{proposition}

\begin{proof}
See Appendix \ref{appendix:a}.
\end{proof}

From Propositions \ref{prop:length} and \ref{prop:determinant}, the number of quadratic constraints on a pure common lines matrix equals ${n \choose 2} + 2{n \choose 3}$.

\medskip

\begin{example}
Consider $n = 4$. Then a pure common lines matrix written in block form,
\begin{equation*}
A = \begin{pmatrix}
\mathbf{0} & \mathbf{a}_{12} & \mathbf{a}_{13} & \mathbf{a}_{14}  \\
\mathbf{a}_{21} & \mathbf{0} & \mathbf{a}_{23} & \mathbf{a}_{24}  \\
\mathbf{a}_{31} & \mathbf{a}_{32} & \mathbf{0} & \mathbf{a}_{34} \\
\mathbf{a}_{41} & \mathbf{a}_{42} & \mathbf{a}_{43} & \mathbf{0}
\end{pmatrix}
\end{equation*}
has rank at most 3 and satisfies 14 quadratic equations, which are
\begin{gather}
\begin{aligned}
    \|\mathbf{a}_{12}\|_2^2 &= \|\mathbf{a}_{21}\|_2^2 & \|\mathbf{a}_{14}\|_2^2 &= \|\mathbf{a}_{41}\|_2^2 \\
    \|\mathbf{a}_{13}\|_2^2 &= \|\mathbf{a}_{31}\|_2^2 & \|\mathbf{a}_{24}\|_2^2 &= \|\mathbf{a}_{42}\|_2^2 \\
    \|\mathbf{a}_{23}\|_2^2 &= \|\mathbf{a}_{32}\|_2^2 & \|\mathbf{a}_{34}\|_2^2 &= \|\mathbf{a}_{43}\|_2^2
\end{aligned} \nonumber\\
\text{det}\begin{pmatrix}
        \mathbf{a}_{12} & \mathbf{a}_{13}
    \end{pmatrix} = -\text{det}\begin{pmatrix}
        \mathbf{a}_{21} & \mathbf{a}_{23}
    \end{pmatrix} = \text{det}\begin{pmatrix}
        \mathbf{a}_{31} & \mathbf{a}_{32}
    \end{pmatrix} \nonumber
\\
\text{det}\begin{pmatrix}
        \mathbf{a}_{12} & \mathbf{a}_{14}
    \end{pmatrix} = -\text{det}\begin{pmatrix}
        \mathbf{a}_{21} & \mathbf{a}_{24}
    \end{pmatrix} = \text{det}\begin{pmatrix}
        \mathbf{a}_{41} & \mathbf{a}_{42}
    \end{pmatrix} \nonumber
\\
\text{det}\begin{pmatrix}
        \mathbf{a}_{13} & \mathbf{a}_{14}
    \end{pmatrix} = -\text{det}\begin{pmatrix}
        \mathbf{a}_{31} & \mathbf{a}_{34}
    \end{pmatrix} = \text{det}\begin{pmatrix}
        \mathbf{a}_{41} & \mathbf{a}_{43}
    \end{pmatrix} \nonumber
\\
\text{det}\begin{pmatrix}
        \mathbf{a}_{23} & \mathbf{a}_{24}
    \end{pmatrix} = -\text{det}\begin{pmatrix}
        \mathbf{a}_{32} & \mathbf{a}_{34}
    \end{pmatrix} = \text{det}\begin{pmatrix}
        \mathbf{a}_{42} & \mathbf{a}_{43}
    \end{pmatrix} \nonumber
\end{gather}
where $\| \cdot \|_2$ denotes the Euclidean norm.
Note that $\text{rank}(A) \leq 3$ is equivalent to the vanishing of all $4 \times 4$ minors of $A$, giving a collection of homogeneous degree $4$ polynomial constraints on the entries of $A$. 
\end{example}

\medskip

We note that \eqref{eq:A-factorization} furnishes a polynomial map which sends an $n$-tuple of rotations to a pure common lines matrix: 
\begin{equation}{\label{phi-map}}
\begin{split}
\psi: \operatorname{SO}(3)^n &\longrightarrow \mathbb{R}^{2n \times n} \\
    (R^{(1)}, \ldots, R^{(n)}) &\mapsto A
\end{split}
\end{equation}
Studying $\psi$ will allow us understand additional important properties of pure common lines matrices. To do this, we will need to introduce some terminology and elementary concepts from algebraic geometry (see \cite{harris1992algebraic} for precise definitions.)

A subset $X \subseteq \mathbb{R}^d$ is called an \textit{algebraic variety} if it is the set of points in $\mathbb{R}^d$ where a finite collection of polynomials all simultaneously equal 0. For example, $SO(3)$ is an algebraic variety since it is the set of matrices $R$ in $\mathbb{R}^{3 \times 3} \cong \mathbb{R}^9$ satisfying the polynomial equations $R^\top R - I_{3 \times 3} = 0_{3 \times 3}$ and $\text{det}(R) - 1 = 0$. 
Roughly speaking, an algebraic variety is similar to an embedded manifold, expect possibly singular and always defined by polynomial equations. 
Due to the properties of polynomials, an algebraic variety $X$ is a ``thin'' subset of $\mathbb{R}^d$ in which it lives: provided $X \neq \mathbb{R}^d$, the complement of $X$ is always a dense subset filling up almost the entirety of the ambient space. 
More precisely, if one samples a random point from $\mathbb{R}^d$ according to \text{any} absolutely continuous probability distribution, then with probability $1$ the point will lie in the complement of $X$. 
We say that some property $P$ holds \textit{(Zariski) generically} if it holds for all points in the complement of \text{some} algebraic variety $X \subsetneq \mathbb{R}^d$, and we call such points \textit{(Zariski) generic}. Roughly speaking, this means that property $P$ holds with probability $1$ (even if, as usually the case, the variety $X \subsetneq \mathbb{R}^d$ is left unspecified).

Recall that the \textit{fiber} of a map at a point $p$ in its image is the set of points in its domain which map to $p$. Therefore to answer the question, ``Does a pure common lines matrix uniquely determine the rotations which generated it?'', we need to understand the fibers of the map $\psi$ in \eqref{phi-map}. Our next result shows that the answer to the question is ``Yes", up to a global rotation, provided the pure common lines matrix is generic.

\medskip

\begin{theorem}{\label{thm:unique-rotations}}
For $n \geq 3$, generically, the fibers of the map $\psi$ are isomorphic to $\operatorname{SO}(3)$. More precisely, for generic rotations $R^{(1)}, \ldots, R^{(n)}$ it holds that 
$$\psi^{-1}(\psi(R^{(1)}, \ldots, R^{(n)})) = \{(R^{(1)}Q, \ldots, R^{(n)}Q) : Q \in \operatorname{SO}(3)\}.$$
\end{theorem}

\begin{proof}
See Appendix \ref{appendix:a}.
\end{proof}

\begin{remark}\label{rem:careful}
 \Cref{thm:unique-rotations} does not contradict the chirality ambiguity in cryo-EM, which states that the 3D molecule and rotations can only ever be recovered up to a global rotation \textit{and} global reflection given cryo-EM data.  
In Theorem \ref{thm:row-column-scales} we prove that there are \textit{two} possible pure common lines matrices for a given non-pure common lines matrix.  They differ by a global sign, and correspond to rotation tuples $(R^{(1)}, \ldots, R^{(n)})$ and $(JR^{(1)}, \ldots, JR^{(n)})$ where $J = \operatorname{diag}(-1,-1,1)$, respectively.
The chirality or handedness ambiguity is well-known in the common lines literature~\cite{VANHEEL1987111,shkolniskySinger1} and unavoidable.
\end{remark}

\subsection{The common lines variety}{\label{sec:common-lines-variety}}
In general, the image of a polynomial map from an algebraic variety is not an algebraic variety, because polynomial \textit{inequalities} (in addition to equations) are needed in the description of the image \cite{bochnak2013real}. This means that the set of all pure common lines matrices, that is, the image of $\psi$ from $SO(3)$, on its own is not an algebraic variety. 
To resolve this, we consider the smallest algebraic variety in $\mathbb{R}^{2n \times n}$ containing $\psi(SO(3))$, i.e. we add the smallest set of additional points (which are not pure common lines matrices) to $\psi(SO(3))$ until the union becomes an algebraic variety. The process is called taking the \textit{Zariski closure} of $\psi(SO(3))$. We call the resulting algebraic variety the \textit{\textbf{common lines variety}} and it lives in the ambient space $\mathbb{R}^{2n \times n}$. 
The common lines variety is defined by polynomial equations in the entries of a matrix $A \in \mathbb{R}^{2n \times n}$. Since the common lines variety includes all pure common lines matrices, the polynomials defining it in particular include the constraints we already identified in \Cref{sec:common-matrix}.

\medskip

\begin{example}\label{ex:common-var}
When $n = 3$, we used the computer algebra system Macaulay2 \cite{M2} to determine the collection of polynomial equations defining the common lines variety.  
Along with the 5 quadratic polynomials from Propositions~\ref{prop:length} and~\ref{prop:determinant}, our computation also found 1 polynomial of degree 6, 64 polynomials of degree 8, and 24 polynomials of degree 10, for a total of 94 polynomial equations. Notice that for  $n=3$, the rank $3$ constraint of \Cref{prop:rank} is vacuous. We find the 5 quadratics are the only homogeneous polynomials. 
The other 89 equations are highly complex and we refrain from explicitly writing them here.  They are available at the Github repository \eqref{link:github}. 
\end{example}

\medskip

In view of \Cref{ex:common-var}, we believe that \Cref{sec:common-matrix}  identifies all  ``simple-to-describe" algebraic constraints on pure common lines matrices.  
As such, it is important to understand to what extent these constraints are enough to characterize pure common lines matrices. This requires understanding the geometry of the common lines variety better, for which we will need to use a couple more basic concepts from algebraic geometry described in the next two paragraphs.

In general, every algebraic variety $X \subseteq \mathbb{R}^{d}$ admits a unique decomposition into a finite union of \textit{irreducible components} $X = \bigcup_{i = 1}^r X_i$, 
where $X_i \subseteq \mathbb{R}^d$ are algebraic varieties themselves and each cannot be decomposed as a union of two strictly smaller varieties.  We think of $X_i$ as ``building blocks'' of $X$.  
For example, the variety $\{(x,y) \in \mathbb{R}^2 : xy = 0\}$ is a union of the $x$- and $y$-axes, and these lines are its irreducible components.
In general, each irreducible component $X_i$ can be ascribed a  \textit{dimension}, which  captures the number of degrees of freedom in $X_i$ and coincides with manifold dimension when $X_i$ is smooth. 
We note that the dimension of different irreducible components $X_i$ of $X$ may differ.  

Given an algebraic variety $X \subseteq \mathbb{R}^d$, one can construct from it a larger algebraic variety $C(X) \subseteq \mathbb{R}^d$ called the \textit{cone} over $X$ by adding to $X$ all points in $\mathbb{R}^d$ which lie on a line passing through the origin and a point on $X$. This constructs an algebraic variety that includes all scalar multiples of points of $X$.

Since the constraints we identified in \Cref{sec:common-matrix} are all polynomial equations, they define \Cref{sec:common-matrix} an algebraic variety in $\mathbb{R}^{2n \times n}$.
In the next proposition we show that this algebraic variety contains the cone over the common lines variety as an irreducible component. This means that locally on this component, our constraints from \Cref{sec:common-matrix} are sufficient to characterize pure common lines matrices up to scale. The proof relies on computer algebra software \cite{M2}, and checking that certain numerical matrices are full-rank, so we also report the range of $n$ on which the proposition has been confirmed.

\medskip

\begin{proposition} \label{prop:suffice-locally}
The algebraic variety defined by the low-rank constraint in \Cref{prop:rank}, along with the quadratic constraints in Propositions~\ref{prop:length} and~\ref{prop:determinant}, and the requirement that all diagonal blocks are $0$, contains the cone over the common lines variety in $\mathbb{R}^{2n \times n}$ as an irreducible component for $n = 3, \ldots, 50$.
\end{proposition}

\begin{proof}
    See Appendix \ref{appendix:a}.
\end{proof}

\section{Optimization problem}

We encode the common lines from cryo-EM data by choosing representatives $\widehat{\mathbf{a}}_{ij} \in \mathbb{R}^2$ (\Cref{def:representatives}) to form the $2 \times 1$ blocks in a common lines matrix $\widehat{A} \in \mathbb{R}^{2n \times n}$.  
Suppose we have rescaled the $2 \times 1$ blocks of $\widehat{A}$ so they all have norm 1. Then at least in clean situations, \Cref{prop:rank} and Propositions \ref{prop:length} and \ref{prop:determinant} imply we can scale the blocks by nonzero scalars $\lambda_{ij}$ with $\lambda_{ij} = \lambda_{ji}$ so that the resulting matrix is a \textit{pure} common lines matrix $A$, and thus has rank 3 and satisfies the set of norm and $2 \times 2$ determinant equations. 
\Cref{prop:suffice-locally} states that these constraints  are sufficient to determine the common lines variety locally.
In Section \ref{sec:correctness}, we further prove that for purposes of recovering scales $\lambda_{ij}$ to obtain a \text{pure} common lines matrix, the constraints are also sufficient.

Proper scales are not directly available from common lines data in cryo-EM.  To find the scales we formulate an optimization problem,  inspired by work for a mathematically similar problem in  \cite{scalingAmit}:
\begin{subequations}{\label{eq:opt-problem}}
\begin{align}
\min_{\substack{\{\mathbf{a}_{ij}\},\{\lambda_{ij}\} \\ i,j = 1, \ldots, n}} \quad & \sum_{i,j = 1}^n \|\widehat{\mathbf{a}}_{ij} - \lambda_{ij}\mathbf{a}_{ij}\|_2 {\tag{\ref{eq:opt-problem}}} \\
\textrm{subject to} \quad\quad & \begin{cases} \mathbf{a}_{ii} = {0} \text{ for all } 1 \leq i \leq n \\
 \mathrm{rank}(A) = 3 \\
 {\lambda_{ij} = \lambda_{ji}  \text{ for all } 1 \leq i < j \leq n},\end{cases} {\label{eq:opt-const-1}} \\
&   \begin{cases}  \|\mathbf{a}_{ij}\|_2^2 = \|\mathbf{a}_{ji}\|_2^2 \\
   \text{det}\begin{pmatrix}
        \mathbf{a}_{ij} & \mathbf{a}_{ik}
    \end{pmatrix} = -\text{det}\begin{pmatrix}
        \mathbf{a}_{ji} & \mathbf{a}_{jk}
    \end{pmatrix} = \text{det}\begin{pmatrix}
        \mathbf{a}_{ki} & \mathbf{a}_{kj}
    \end{pmatrix} \\
 \text{for all } 1 \leq i < j < k \leq n.
\end{cases} {\label{eq:opt-const-2}}
\end{align}
\end{subequations}

\medskip

\noindent The mixed L1/Frobenius norm $\|\cdot\|_2$ in the objective is chosen for its robustness to outliers.

Once we obtain a pure common lines matrix, we show in Section \ref{sec:rot-recovery} how to recover the rotations corresponding to the common lines (up to the ambiguity in \Cref{rem:careful}). 
Later in Section \ref{sec:heterogeneity-application}, we solve the problem \eqref{eq:opt-problem} to identify homogeneous clusters among images coming from a discrete number of distinct molecules. 

\section{Optimization algorithms}\label{sec:algo}

Our approach to solving \eqref{eq:opt-problem} is first to solve the problem with constraint \eqref{eq:opt-const-1} only, and then to enforce the constraint \eqref{eq:opt-const-2} on the solution. These steps are in Sections \ref{sec:admm-irls} and \ref{sec:sinkhorn} respectively.

\subsection{IRLS and ADMM for the rank constraint}{\label{sec:admm-irls}}

To solve \eqref{eq:opt-problem} with the rank constraint, we closely follow the approach of ~\cite{scalingAmit}. We relax the mixed L1/Frobenius norm to a weighted least squares objective, where the weights and optimization variables are updated after each iteration of minimization via a procedure called Iterative Reweighted Least Squares (IRLS) \cite{IRLS}. Let $t$ denote the IRLS iteration number. Then the objective \eqref{eq:opt-problem} becomes:
\begin{align}{\label{eq:opt-IRLS}}
\min_{A \in \mathbb{R}^{2n \times n}, \; \Lambda \in \mathbb{R}^{n \times n}} \quad & \|\widehat{A} - (\Lambda \otimes \mathbf{1}_{2 \times 1}) \odot A\|_{WF}^2
\end{align}
where $\otimes$ and $\odot$ is the Kronecker and Hadamard product of two matrices respectively, $\Lambda_{ij} = \lambda_{ij}$,
\begin{equation*}
\|M\|_{WF}^2 := \sum_{i,j = 1}^n w_{ij}^{(t)}\|\mathbf{m}_{ij}\|_2^2
\end{equation*}
is the squared weighted Frobenius norm of a block matrix $M \in \mathbb{R}^{2n \times n}$, the weights in \eqref{eq:opt-IRLS} are
\begin{equation}{\label{eq:irls-weights}}
w_{ij}^{(t)} = \begin{cases}
    1/(\max\{\delta,\|\widehat{\mathbf{a}}_{ij} - \lambda_{ij}^{(t-1)}\mathbf{a}_{ij}^{(t-1)}\|_2\}) &\quad \text{if } i \neq j \\
    0 &\quad \text{if } i = j,
\end{cases}
\end{equation}
\noindent and $0 < \delta \ll 1$ is a chosen regularization parameter.

Within each iteration of IRLS, we need to solve the problem \eqref{eq:opt-IRLS} with the constraints \eqref{eq:opt-const-1}. Since the objective is bilinear in $A$ and $\Lambda$, we can do this using the Alternate Direction Method of Multiplier (ADMM) \cite{boyd2011distributed, admm-goldstein}. This gives an augmented Lagrangian optimization problem:
\begin{equation}{\label{eq:admm}}
\begin{aligned}
\max_{\Gamma \in \mathbb{R}^{2n \times n}} \quad \min_{A,B \in \mathbb{R}^{2n \times n}, \; \Lambda \in \mathbb{R}^{n \times n}} \quad & \frac{1}{2}\|\widehat{A} - (\Lambda \otimes \mathbf{1}_{2 \times 1}) \odot A\|_{WF}^2 + \frac{\tau}{2}\|B - (A + \Gamma)\|_{F}^2 \\
\textrm{subject to} \quad\quad & \begin{cases} \mathbf{a}_{ii} = \mathbf{0} \text{ for all } 1 \leq i \leq n \\
\Lambda = \Lambda^\top \\
\mathrm{rank}(B) = 3, \end{cases}
\end{aligned}
\end{equation}
where $\Gamma$ is a matrix of Lagrange multipliers and $\tau = \sum_{i,j=1}^n w_{ij}^{(t)}$. We now describe the steps of the ADMM procedure. Since the problem \eqref{eq:admm} is non-convex, we alternatingly optimize for each variable. In the following, let $k$ denote the ADMM iteration number and let $W^{(t)} \in \mathbb{R}^{n \times n}$ where $W_{ij}^{(t)} = w_{ij}^{(t)}$ be the matrix of weights within IRLS iteration $t$.

\medskip

\noindent \textbf{1. Optimize $A$ and $\Lambda$:} We alternatingly optimize for $A$ and $\Lambda$ until convergence. Let $k^\prime$ denote the iteration number for this step.
\begin{enumerate}
    \item[1a.] First we solve the unconstrained problem for $A$:
    \begin{align*}
 \min_{A \in \mathbb{R}^{2n \times n}} & \quad \frac{1}{2}\|\widehat{A} - (\Lambda^{(k^\prime)} \otimes \mathbf{1}_{2 \times 1}) \odot A\|_{WF}^2 + \frac{\tau}{2}\|B^{(k)} - (A + \Gamma^{(k)})\|_F^2
\end{align*}
The solution is
\begin{equation}{\label{eq:admm-A}}
    \begin{aligned}
        A^{(k^\prime+1)} &= \left((W^{(t)} \otimes \mathbf{1}_{2 \times 1}) \odot (\Lambda^{(k^\prime)} \otimes \mathbf{1}_{2 \times 1}) \odot \widehat{A} + \frac{\tau}{4}(B^{(k)} + \Gamma^{(k)})\right) \\
        & \oslash\left((W^{(t)} \otimes \mathbf{1}_{2 \times 1}) \odot (\Lambda^{(k^\prime)} \otimes \mathbf{1}_{2 \times 1}) \odot (\Lambda^{(k^\prime)} \otimes \mathbf{1}_{2 \times 1}) + \frac{\tau}{4}\textbf{1}_{2n \times n}\right)
    \end{aligned}
\end{equation}
    where $\oslash$ is the element-wise division of two matrices. Then we project $A$ onto the set of matrices whose $2 \times 1$ diagonals are 0:
    \begin{equation}{\label{eq:admm-A-zero}}
        \mathbf{a}^{(k^\prime+1)}_{ii} = \mathbf{0}
    \end{equation}
    \item[1b.] Next we solve the unconstrained problem for $\Lambda$:
\begin{equation*}
\begin{aligned}
\min_{\Lambda \in \mathbb{R}^{n \times n}} & \quad \frac{1}{2}\|\widehat{A} - (\Lambda \otimes \mathbf{1}_{2 \times 1}) \odot A^{(k^\prime+1)}\|_{WF}^2 \\
\textrm{subject to} &\quad  \Lambda = \Lambda^\top
\end{aligned}    
\end{equation*}

The solution is 
\begin{equation}{\label{eq:admm-lambda}}
        \lambda_{ij}^{(k^\prime+1)} = \begin{cases}
            \frac{w_{ij}\langle \widehat{\mathbf{a}}_{ij}, \mathbf{a}_{ij}^{(k^\prime+1)} \rangle + w_{ji}\langle \widehat{\mathbf{a}}_{ji}, \mathbf{a}_{ji}^{(k^\prime+1)} \rangle}{w_{ij}\|\mathbf{a}_{ij}^{(k^\prime+1)}\|_2^2+ w_{ji}\|\mathbf{a}_{ji}^{(k^\prime+1)}\|_2^2} &\quad \text{ if } i \neq j \\[1em]
            0 &\quad \text{ if } i = j.
        \end{cases}
    \end{equation}

\end{enumerate}
After repeating 1a. and 1b. until convergence, we obtain $A^{(k+1)}$ and $\Lambda^{(k+1)}$.

\medskip

{\label{admm-new-B}}

\noindent \textbf{2. Optimize $B$:} The constrained problem for $B$ is
\begin{align*}
 \min_{B \in \mathbb{R}^{2n \times n}} & \quad \frac{\tau}{2}\|B - (A^{(k+1)} + \Gamma^{(k)})\|_F^2  \\
\textrm{subject to} & \quad \mathrm{rank}(B) = 3
\end{align*}
This is solved by
\begin{equation}{\label{eq:admm-B}}
    B^{(k+1)} = SVP(A^{(k+1)} - \Gamma^{(k)},3)
\end{equation}
where $SVP(M,3)$ is the singular value projection of a matrix $M$ onto the set of matrices of rank at most $3$, which is computing by taking the highest three singular values of $M$ and its corresponding left and right singular vectors.

\medskip

\noindent \textbf{3. Update $\Gamma$:} In ADMM, there is a gradient ascent step for $\Gamma$, where the step is the solution to
\begin{equation*}
    \max_{\Gamma \in \mathbb{R}^{2n \times n}} \|B^{(k+1)} - (A^{(k+1)} - \Gamma)\|_F^2.
\end{equation*}
This gives the update
\begin{equation}{\label{eq:admm-gamma}}
    \Gamma^{(k+1)} = \Gamma^{(k)} + (B^{(k+1)} - A^{(k+1)}).
\end{equation}

Steps 1., 2., and 3. are repeated until convergence in the optimization variables. This completes the ADMM procedure for IRLS iteration $t$. The IRLS weights $w_{ij}^{(t+1)}$ for iteration $t+1$ are updated using \eqref{eq:irls-weights}, and the ADMM procedure is repeated again. 
The whole pipeline is detailed in Algorithm \ref{alg:irls-admm}.

\begin{algorithm}
\caption{IRLS and ADMM for rank constraint satisfaction}{\label{alg:irls-admm}}
\begin{algorithmic}[1]
\Input{$\widehat{A} \in \mathbb{R}^{2n \times n}$, a common lines matrix}
\Output{$A \in \mathbb{R}^{2n \times n}$, a common lines matrix satisfying only the rank constraint}
\Procedure{IRLS-ADMM}{$\widehat{A}$}
\State initialize $A, \Lambda, W$
\State $t \gets 0$
\While{not converged}
    \State $B \gets A$
    \State $\Gamma \gets 0_{2n \times n}$
    \State $\tau \gets \sum_{i,j=1}^n w_{ij}^{(t)}$
    \State $k \gets 0$
    \While{not converged}
        \State $k' \gets 0$
        \While{not converged}
            \State update $A$ using \eqref{eq:admm-A} and \eqref{eq:admm-A-zero} \Comment{1. Update $A$ and $\Lambda$}
            \State update $\Lambda$ using \eqref{eq:admm-lambda}
            \State $k' \gets k' + 1$
        \EndWhile
        \State update $B$ using \eqref{eq:admm-B} \Comment{2. Update $B$}
        \State update $\Gamma$ using \eqref{eq:admm-gamma} \Comment{3. Update $\Gamma$}
        \State $k \gets k + 1$
    \EndWhile
    \State update $W$ using \eqref{eq:irls-weights} \Comment{Update IRLS weights}
    \State $t \gets t + 1$
\EndWhile
\EndProcedure
\end{algorithmic}
\end{algorithm}

\subsection{Sinkhorn scaling for the quadratic constraints}{\label{sec:sinkhorn}}
When successful, \textsc{IRLS-ADMM} in \Cref{sec:admm-irls} gives us a solution  $A$ to \eqref{eq:opt-problem} satisfying the constraint \eqref{eq:opt-const-1}.
Next we must enforce \eqref{eq:opt-const-2}. 
As described below, our approach is to scale the rows and columns of $A$ alternatingly until constraint \eqref{eq:opt-const-2} is satisfied, in a manner analogous to Sinkhorn's algorithm \cite{sinkhorn1967}. Note that nonzero row and column scales do not affect the rank of $A$, so constraint \eqref{eq:opt-const-1} will still be satisfied.

We find the row and column scales by solving least squares problems. First we handle the norm constraints. 
Define $M \in \mathbb{R}^{n \times n}$ where
\begin{equation}{\label{eq:M-matrix-symmetry}}
M_{ij} = \|\mathbf{a}_{ij}\|_2^2   
\end{equation}

\medskip

\noindent Then the norm constraints are satisfied if and only if $M = M^\top$, which leads us to the following constrained least squares problems:
\begin{equation}{\label{sinkhorn-d1}}
    \boldsymbol{\mu} = \argmin_{\|\boldsymbol{\mu}\|_2=1} \|\text{diag}(\boldsymbol{\mu})M - (\text{diag}(\boldsymbol{\mu})M)^\top\|_F^2
\end{equation}
\begin{equation}{\label{sinkhorn-d2}}
    \boldsymbol{\tau} = \argmin_{\|\boldsymbol{\tau}\|_2=1} \|M\text{diag}(\boldsymbol{\tau}) - (M\text{diag}(\boldsymbol{\tau}))^\top\|_F^2
\end{equation}

\noindent

\noindent The solutions to problems \eqref{sinkhorn-d1} and \eqref{sinkhorn-d2} are
\begin{equation}{\label{D1-minimization}}
    \min_{\|\boldsymbol{\mu}\|_2=1} \left\|N_L  \boldsymbol{\mu}\right\|_2^2,
\end{equation}
\begin{equation}{\label{D2-minimization}}
    \min_{\|\boldsymbol{\tau}\|_2=1} \left\|N_R \boldsymbol{\tau}\right\|_2^2
\end{equation}
respectively, where $N_L, N_R \in \mathbb{R}^{n \times n}$ are the corresponding least squares matrices. See Appendix \ref{appendix:b} for full detail. The problems \eqref{D1-minimization} and \eqref{D2-minimization} are solved by taking $\boldsymbol{\mu}$ and $\boldsymbol{\tau}$ to be the right singular vector corresponding to the smallest singular value of $N_L$ and $N_R$ respectively.

Now we handle the determinant constraints. Scaling each $2 \times n$ row of $A$ by $\mu_1,\ldots,\mu_n \in \mathbb{R}$ and enforcing the constraints leads us to the equations
\begin{equation}{\label{eq:determinant-row-scalings}}
    \mu_i^2\text{det}\begin{pmatrix}
        \mathbf{a}_{ij} & \mathbf{a}_{ik}
    \end{pmatrix} = -\mu_j^2\text{det}\begin{pmatrix}
        \mathbf{a}_{ji} & \mathbf{a}_{jk}
    \end{pmatrix} = \mu_k^2\text{det}\begin{pmatrix}
        \mathbf{a}_{ki} & \mathbf{a}_{kj}
    \end{pmatrix}
\end{equation}
for all $1 \leq i < j < k \leq n$. Taking the signed root on each equation, we obtain
\begin{equation}{\label{eq:determinant-scales-rows}}
\begin{aligned}
    \mu_i\text{sgn}(\text{det}\begin{pmatrix}
        \mathbf{a}_{ij} & \mathbf{a}_{ik}
    \end{pmatrix})\sqrt{|\text{det}\begin{pmatrix}
        \mathbf{a}_{ij} & \mathbf{a}_{ik}
    \end{pmatrix}|} &= -\mu_j\text{sgn}(\text{det}\begin{pmatrix}
        \mathbf{a}_{ji} & \mathbf{a}_{jk}
    \end{pmatrix})\sqrt{|\text{det}\begin{pmatrix}
        \mathbf{a}_{ji} & \mathbf{a}_{jk}
    \end{pmatrix}|} \\ &= \mu_k\text{sgn}(\text{det}\begin{pmatrix}
        \mathbf{a}_{ki} & \mathbf{a}_{kj}
    \end{pmatrix})\sqrt{|\text{det}\begin{pmatrix}
        \mathbf{a}_{ki} & \mathbf{a}_{kj}
    \end{pmatrix}|}
\end{aligned}
\end{equation}
Scaling the columns of $A$ by $\tau_1,\ldots,\tau_n \in \mathbb{R}$ and enforcing the constraints leads to the equations
\begin{equation}{\label{eq:determinant-column-scalings}}
    \tau_j\tau_k\text{det}\begin{pmatrix}
        \mathbf{a}_{ij} & \mathbf{a}_{ik}
    \end{pmatrix} = -\tau_i\tau_k\text{det}\begin{pmatrix}
        \mathbf{a}_{ji} & \mathbf{a}_{jk}
    \end{pmatrix} = \tau_i\tau_j\text{det}\begin{pmatrix}
        \mathbf{a}_{ki} & \mathbf{a}_{kj}
    \end{pmatrix}
\end{equation}
for all $1 \leq i < j < k \leq n$. Dividing the first equation above by $\tau_k$ on both sides and the second equation above by $\tau_i$ on both sides, we obtain
\begin{equation}{\label{eq:determinant-scales-columns}}
    \begin{aligned}
        \tau_j\text{det}\begin{pmatrix}
        \mathbf{a}_{ij} & \mathbf{a}_{ik}
    \end{pmatrix} &= -\tau_i\text{det}\begin{pmatrix}
        \mathbf{a}_{ji} & \mathbf{a}_{jk}
    \end{pmatrix} \\
    -\tau_k\text{det}\begin{pmatrix}
        \mathbf{a}_{ji} & \mathbf{a}_{jk}
    \end{pmatrix} &= \tau_j\text{det}\begin{pmatrix}
        \mathbf{a}_{ki} & \mathbf{a}_{kj}
    \end{pmatrix}
    \end{aligned}
\end{equation}
We observe that equations \eqref{eq:determinant-scales-rows} and \eqref{eq:determinant-scales-columns} are linear in $\mu_i$ and $\tau_i$. We can enumerate all determinants into three vectors
\begin{equation}{\label{eq:v-vectors-determinant}}
\begin{aligned}
    \mathbf{v}_1 &= \begin{pmatrix}
\text{det}\begin{pmatrix}
        \mathbf{a}_{ij} & \mathbf{a}_{ik}
    \end{pmatrix}   
\end{pmatrix}_{1 \leq i < j < k \leq n} \\
    \mathbf{v}_2 &= \begin{pmatrix}
-\text{det}\begin{pmatrix}
        \mathbf{a}_{ji} & \mathbf{a}_{jk}
    \end{pmatrix}
\end{pmatrix}_{1 \leq i < j < k \leq n} \\
    \mathbf{v}_3 &= \begin{pmatrix}
\text{det}\begin{pmatrix}
        \mathbf{a}_{ki} & \mathbf{a}_{kj}
    \end{pmatrix}
\end{pmatrix}_{1 \leq i < j < k \leq n}
\end{aligned}
\end{equation}
\noindent each of length ${n \choose 3}$. The determinant constraints are satisfied if and only if $\mathbf{v}_1 = \mathbf{v}_2 = \mathbf{v}_3$, which leads to the following constrained least squares problems: 
\begin{align}{\label{determinant-scales-LS-1}}
 \min_{\|\boldsymbol{\mu}\|_2=1} \|(\boldsymbol{\mu} \;\triangle\; \mathbf{v}_1) - (\boldsymbol{\mu} \;\triangle\; \mathbf{v}_2)\|_2^2 + \|(\boldsymbol{\mu} 
 \;\triangle\; \mathbf{v}_2) - (\boldsymbol{\mu} \;\triangle\; \mathbf{v}_3)\|_2^2
\end{align}
\begin{align}{\label{determinant-scales-LS-2}}
 \min_{\|\boldsymbol{\tau}\|_2=1} \quad \|(\boldsymbol{\tau} \;\triangle_1\; \mathbf{v}_1) - (\boldsymbol{\tau} \;\triangle_1\; \mathbf{v}_2)\|_2^2 + \| (\boldsymbol{\tau} \;\triangle_2\; \mathbf{v}_2) - (\boldsymbol{\tau} \;\triangle_2\; \mathbf{v}_3)\|_2^2
\end{align}
where the quantities in brackets (defined in \eqref{eq:mu-triangle} and \eqref{eq:tau-triangle}) are the corresponding scalings \eqref{eq:determinant-scales-rows} and \eqref{eq:determinant-scales-columns} of $\mathbf{v}_1$, $\mathbf{v}_2$, and $\mathbf{v}_3$ by $\boldsymbol{\mu}$ and $\boldsymbol{\tau}$. The solutions to problems \eqref{determinant-scales-LS-1} and \eqref{determinant-scales-LS-2} are
\begin{equation}{\label{det1-minimization}}
    \min_{\| \boldsymbol{\mu}\|_2=1} \left\|(D_{L,1} + D_{L,2})  \boldsymbol{\mu}\right\|_2^2
\end{equation}
\begin{equation}{\label{det2-minimization}}
    \min_{\|\boldsymbol{\tau}\|_2 =1} \left\|(D_{R,1} + D_{R,2})  \boldsymbol{\tau}\right\|_2^2
\end{equation}
respectively, where $D_{L,1}, D_{L,2}, D_{R,1}, D_{R,2} \in \mathbb{R}^{n \times n}$ are the corresponding least squares matrices. See Appendix \ref{appendix:b} for full detail. The problems \eqref{det1-minimization} and \eqref{det2-minimization} are again solved using SVD.

Now we describe the steps of the Sinkhorn scaling method. Let $r$ denote the iteration number of the procedure. 

\medskip

\noindent \textbf{1. Scale rows:} Let $\boldsymbol{\mu} \in \mathbb{R}^n$ be the solution to
\begin{equation}{\label{all1-minimization}}
    \min_{\|\boldsymbol{\mu}\|_2=1} \left\|(N_L + D_{L,1} + D_{L,2}) \boldsymbol{\mu}\right\|_2^2
\end{equation}
Then we perform the update
\begin{equation}{\label{eq:update-rows-A}}
    A^{(r+\frac{1}{2})} = \frac{\|A^{(r)}\|_F^2}{\|(\text{diag}(\boldsymbol{\mu}) \otimes \mathbf{1}_{2 \times 1})A^{(r)}\|_F^2}(\text{diag}(\boldsymbol{\mu}) \otimes \mathbf{1}_{2 \times 1})A^{(r)}
\end{equation}

\noindent \textbf{2. Scale columns:} Let $\boldsymbol{\tau} \in \mathbb{R}^n$ be the solution to
\begin{equation}{\label{all2-minimization}}
    \min_{\| \boldsymbol{\tau}\|_2=1} \left\|(N_R + D_{R,1} + D_{R,2})  \boldsymbol{\tau}\right\|_2^2
\end{equation}
Then we perform the update
\begin{equation}{\label{eq:update-columns-A}}
    A^{(r+1)} = \frac{\|A^{(r+\frac{1}{2})}\|_F^2}{\|A^{(r+\frac{1}{2})}\text{diag}(\boldsymbol{\tau})\|_F^2}A^{(r+\frac{1}{2})}\text{diag}(\boldsymbol{\tau})
\end{equation}
The Sinkhorn scaling procedure is detailed in Algorithm \ref{alg:sinkhorn}.

\begin{algorithm}
\caption{Sinkhorn scaling for quadratic constraint satisfaction}{\label{alg:sinkhorn}}

\begin{algorithmic}[1]
\Input
{$A \in \mathbb{R}^{2n \times n}$, the output of \textsc{IRLS-ADMM}}
\Output
{$A^\prime \in \mathbb{R}^{2n \times n}$, a pure common lines matrix}
\Procedure{\textsc{Sinkhorn}}{$A$}
\State $r \gets 0$
\While{not converged}
\State set $N_L,D_{L,1},D_{L,2}$ using \eqref{eq:Nl-solution}, \eqref{eq:DL1-solution}, \eqref{eq:DL2-solution}
\State set $\boldsymbol{\mu}$ to be the solution of \eqref{all1-minimization}
\State update $A$ using \eqref{eq:update-rows-A} \Comment{1. Scale rows}
\State set $N_R,D_{R,1},D_{R,2}$ using \eqref{eq:Nr-solution}, \eqref{eq:DR1-solution}, \eqref{eq:DR2-solution}
\State set $\boldsymbol{\tau}$ to be the solution of \eqref{all2-minimization}
\State update $A$ using \eqref{eq:update-columns-A} \Comment{2. Scale columns}
\State $r \gets r + 1$
\EndWhile
\State $A^\prime \gets A$
\EndProcedure
\end{algorithmic}
\end{algorithm}

\subsection{Rotation recovery}{\label{sec:rot-recovery}}

\textsc{IRLS-ADMM} and \textsc{Sinkhorn} aim to output a pure common lines matrix $A \in \mathbb{R}^{2n \times n}$. 
Given a pure common lines matrix, we now show how to determine the underlying rotations $R^{(i)}$ in \eqref{eq:restriction} which generated $A$ (recall \Cref{thm:unique-rotations}).

Given $A$, use singular value decomposition to compute a rank-$3$ factorization $A = B C^{\top}$ for $B \in \mathbb{R}^{2n \times 3}$ and $C \in \mathbb{R}^{n \times 3}$.
We then seek an invertible matrix $Q \in \mathbb{R}^{3 \times 3}$ 
so that $BQ$ and $CQ^{- \top}$ take the form of the factors in \eqref{eq:A-factorization}.
In particular, it is required that the rows of $BQ$ come in orthonormal pairs; more precisely, $(BQ)(BQ)^{\top} = B QQ^{\top} B^{\top} \in \mathbb{R}^{2n \times 2n}$ should have $2 \times 2$ identities along its diagonal.
Setting $X = QQ^{\top} \in \mathbb{R}^{3 \times 3}$ and relaxing positive semidefiniteness, let us consider the following affine-linear least squares problem:
\begin{equation}\label{eq:my-least-squares}
\min_{\substack{X \in \mathbb{R}^{3 \times 3} \\ X = X^{\top}}} \| \mathcal{L}(B X B^{\top}) \|_F^2,
\end{equation}
where $\mathcal{L} : \mathbb{R}^{2n \times 2n} \rightarrow \mathbb{R}^{2n \times 2n}$ denotes the affine-linear operator which sets all entries of a $2n \times 2n$ matrix outside of the $2 \times 2$ diagonal blocks to 0, and subtracts the $2 \times 2$ identity matrices from the diagonal blocks. 
The normal equations for \eqref{eq:my-least-squares} may be written as
\begin{equation}{\label{eq:rotation-least-squares}}
    L\operatorname{vec}(X) = \operatorname{vec}(B^\top B),
\end{equation}
where $L \in \mathbb{R}^{9 \times 9}$, whose rows and columns index the variables $X_{ij}$ for $1 \leq i \leq j \leq 3$, and whose corresponding $(ij,k\ell)$-th entry is
\begin{equation}{\label{eq:L-matrix-rotation-recovery}}
    (L)_{ij,k\ell} = \langle (I_{n \times n} \otimes \mathbf{1}_{1 \times 2})(\mathbf{b}_i \odot \mathbf{b}_\ell), (I_{n \times n} \otimes \mathbf{1}_{1 \times 2})(\mathbf{b}_j \odot \mathbf{b}_k) \rangle
\end{equation}
where $\mathbf{b}_i$ is the $i$-th column of $B$.
Generically there is a unique symmetric matrix $X$ solving \eqref{eq:rotation-least-squares}.

Let $X = V D V^\top$ be an eigendecomposition for the solution to \eqref{eq:my-least-squares}, where $V, D \in \mathbb{R}^{9 \times 9}$.  In the clean case, the diagonal matrix $D$ is already positive semidefinite. 
Then $V D^{1/2} \in \mathbb{R}^{3 \times 3}$ is a candidate for $Q$. 
Next the $i$-th row block of $BQ$ determines the first two rows of $R^{(i)}$, and the cross product of these two rows computes the third row of $R^{(i)}$.
From this, we recover the $n$-tuple of rotations $(R^{(1)}, \ldots, R^{(n)})$ up to global right-multiplication by a rotation.
The full procedure for rotation recovery is in Algorithm \ref{alg:rot-recovery}, which is formulated to apply to noisy inputs as well.
We note that as explained in Remark \ref{rem:careful}, a pure common lines matrix $A$ can only be recovered up to sign, because there are two distinct $n$-tuples of rotations that can be recovered corresponding to the chiral ambiguity of common lines data.
One can run \textsc{Rotations} twice, on $A$ and $-A$, to produce the two possible sets of rotations.

\begin{algorithm}
\caption{Rotation recovery}{\label{alg:rot-recovery}}

\begin{algorithmic}[1]
\Input
{$A \in \mathbb{R}^{2n \times n}$, an estimate for pure common lines matrix}
\Output
{$(R_1,\ldots,R_n) \in \operatorname{SO}(3)^n$, rotations determining the pure common lines matrix}
\Procedure{\textsc{Rotations}}{$A$}
\State set $B$ using the SVD to get a rank-3 approximation $BC^\top$ for $A$
\State set $L$ using \eqref{eq:L-matrix-rotation-recovery}
\State set $X$ to be least squares solution of \eqref{eq:rotation-least-squares} 
\State set $V,D$ using the eigendecomposition $X = V D V^\top$ 
\State $D_{ii} \gets \max(D_{ii}, 0)$
\For{$i = 1,\ldots,n$}
\State set $\mathbf{q}_j^\top$ to be the $(2i-2+j)$-th row of $BVD^{1/2}$ for $j = 1,2$
\State set $\mathbf{q}_3^{\top}$ to be the $i$-th row of $C(V D^{1/2})^{-\top}$
\State set $R_i \in \mathbb{R}^{3 \times 3}$ to be $\begin{pmatrix} \mathbf{q}_2^{\top} \\ - \mathbf{q}_1^{\top} \\ \mathbf{q}_3^{\top} \end{pmatrix}$
\EndFor
\If{$\det(R_i) < 0$}
\For{$i = 1,\ldots,n$}
\State $R_i \gets -R_i$
\EndFor
\EndIf
\For{$i = 1,\ldots,n$}
\State replace $R_i$ by the nearest rotation matrix to $R_i$ using the SVD of $R_i$
\EndFor
\EndProcedure
\end{algorithmic}
\end{algorithm}

\subsection{Justification of algorithms}{\label{sec:correctness}}

Suppose the input to \textsc{IRLS-ADMM} is a noiseless common lines matrix $\widehat{A} \in \mathbb{R}^{2n \times n}$, and its output $A \in \mathbb{R}^{2n \times n}$ is a global minimizer to the non-convex problem \eqref{eq:admm}. 
In the next theorem we show that we can recover the ground-truth pure common lines matrix, up to a global scale, by scaling the $2 \times 1$ blocks of $A$ via \textsc{IRLS-ADMM} and enforcing the quadratic constraints via \textsc{Sinkhorn}. The theorem justifies using \textsc{IRLS-ADMM} and \textsc{Sinkhorn}.

\medskip

\begin{theorem}{\label{thm:row-column-scales}}
Let $n \geq 4$.
Let $A \in \mathbb{R}^{2n \times n}$ be a generic pure common lines matrix and $\lambda_{ij} \in \mathbb{R}$ be nonzero scales for $i,j = 1,\ldots,n$ with $i \neq j$. Let $\Lambda \in \mathbb{R}^{n \times n}$ where $\Lambda_{ij} = \lambda_{ij}$ if $i \neq j$, $\Lambda_{ii} = 0$, and $\Lambda = \Lambda^\top$. Suppose
\begin{equation*}
    B := (\Lambda \otimes \mathbf{1}_{2 \times 1}) \odot A = \begin{pmatrix}
\mathbf{0} & \lambda_{12}\mathbf{a}_{12} & \ldots & \lambda_{1,n-1}\mathbf{a}_{1,n-1} & \lambda_{1,n}\mathbf{a}_{1,n}  \\
\lambda_{21}\mathbf{a}_{21} & \mathbf{0} & \ldots & \lambda_{2,n-1}\mathbf{a}_{2,n-1} &  \lambda_{2,n}\mathbf{a}_{2,n}  \\
\vdots & \vdots & \ddots & \vdots & \vdots \\
\lambda_{n-1,1}\mathbf{a}_{n-1,1} & \lambda_{n-1,2}\mathbf{a}_{n-1,2} & \ldots & \mathbf{0} & \lambda_{n-1,n}\mathbf{a}_{n-1,n}  \\
\lambda_{n,1}\mathbf{a}_{n,1} & \lambda_{n,2}\mathbf{a}_{n,2} & \ldots & \lambda_{n,n-1}\mathbf{a}_{n,n-1} & \mathbf{0}
\end{pmatrix}
\end{equation*}
has rank 3. Then there exist $\mu_1,\ldots,\mu_n,\tau_1,\ldots,\tau_n\in \mathbb{R}$ such that
\begin{equation}{\label{eq:row-column-scale}}
\begin{pmatrix}
0 & \lambda_{12} & \ldots & \lambda_{1,n-1} & \lambda_{1,n}  \\
\lambda_{21} & 0 & \ldots & \lambda_{2,n-1} &  \lambda_{2,n}  \\
\vdots & \vdots & \ddots & \vdots & \vdots \\
\lambda_{n-1,1} & \lambda_{n-1,2} & \ldots & 0 & \lambda_{n-1,n}  \\
\lambda_{n,1} & \lambda_{n,2} & \ldots & \lambda_{n,n-1} & 0
\end{pmatrix} = \begin{pmatrix}
0 & \mu_1\tau_1 & \ldots & \mu_1\tau_{n-1} & \mu_1\tau_n  \\
\mu_2\tau_1 & 0 & \ldots & \mu_2\tau_{n-1} &  \mu_2\tau_{n-1}  \\
\vdots & \vdots & \ddots & \vdots & \vdots \\
\mu_{n-1}\tau_1 & \mu_{n-1}\tau_2 & \ldots & 0 & \mu_{n-1}\tau_n  \\
\mu_n\tau_1 & \mu_n\tau_2 & \ldots & \mu_n\tau_{n-1} & 0
\end{pmatrix}
\end{equation}
If $B$ additionally satisfies the quadratic constraints \eqref{eq:opt-const-2}, then there exists $\tau \in \mathbb{R}$ such that for all $i,j = 1,\ldots,n$ ($i \neq j$) it holds $\lambda_{ij} = \tau$.
\end{theorem}
\begin{proof}
See Appendix \ref{appendix:a}.
\end{proof}

\medskip

\begin{remark}{\label{rem:row-signs}}
Theorem \ref{thm:unique-rotations} implies that given a pure common lines matrix, the rotation recovery problem in Section \ref{sec:rot-recovery} is uniquely solvable up to a global rotation.
\Cref{thm:row-column-scales} states that the ground-truth pure common lines matrix can only be determined up to a global scale; in particular, $\tau$ in \Cref{thm:row-column-scales} may be positive or negative. As in Remark \ref{rem:careful}, this sign flip corresponds to chiral ambiguity in cryo-EM.  Apart from its sign, the global scale has no effect on rotation recovery in \Cref{alg:rot-recovery}.
\end{remark}

\section{Application: Clustering heterogeneous common lines}{\label{sec:heterogeneity-application}}

\begin{figure}[!ht]
\centering
\includegraphics[width=0.6\textwidth]{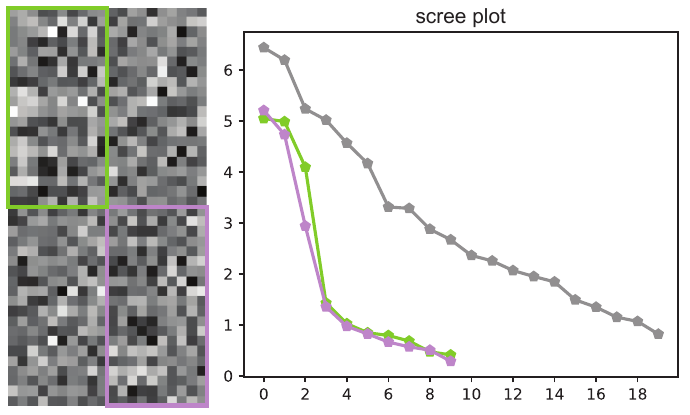}
\vspace{-0.5em}
{\caption{A heterogeneous common lines matrix and rank test for simulated rotations from two distinct molecules. Block-diagonals comparing rotations from the same molecule show rank-3 structure}
\label{fig-svd}}
\end{figure}

Here we present an application of our approach for common lines to a challenging problem in cryo-EM.   
We propose a clustering algorithm for detecting homogeneous communities of consistent common lines from discretely heterogeneous data, using our algebraic constraints.  
One can then use the clusters of common lines for rotation recovery and 3D reconstruction.

Several successful methods have been proposed for clustering heterogeneous cryo-EM data that consist of images of a \textit{single} macromolecule with conformational landscapes or differences in subunits \cite{AIZENBUD2019,HERMAN2008,SHATSKY2010,PENCZEK2006,LIAO2015,SCHERES_2012, Katsevich_2015}.  
Recent work of the third author 
 \cite{verbeke_separating_2020} proposed a method of solving a different heterogeneity challenge in cryo-EM, where the heterogeneity in the data comes from \textit{multiple} distinct macromolecules  rather than variations on one primary structure. Our proposed application focuses on the latter problem.  For this setting of heterogeneity, the main prior work to compare against is \cite{verbeke_separating_2020}.

The basic idea of our approach is illustrated in Figure \ref{fig-svd}. There the matrix in grey is a matrix of common lines from simulated heterogeneous data, corresponding to two distinct molecular conformations.  
The two homogeneous common lines matrices are diagonal blocks in green and purple. The $2 \times 1$ entries outside of these diagonal blocks do not correspond to any consistent lines and are just random lines in $\mathbb{R}^2$, encoded as representatives. 
Note that in general, a heterogeneous common lines matrix will not necessarily have consistent common lines matrices as diagonal blocks, but will be a $2 \times n$ row and $n \times 1$ column permutation of such a matrix. Gaussian white noise has also been added to the grey matrix to decrease the signal-to-noise ratio of the common lines. A scree plot in Figure \ref{fig-svd} shows a noticeable spectral gap between the third and fourth singular values of the homogeneous common lines matrices (green and purple curve), which is not detectable for the entire heterogeneous common lines matrix (grey curve), thus demonstrating low-rank structure of submatrices.

\begin{figure}[H]
\includegraphics[width=1\textwidth]{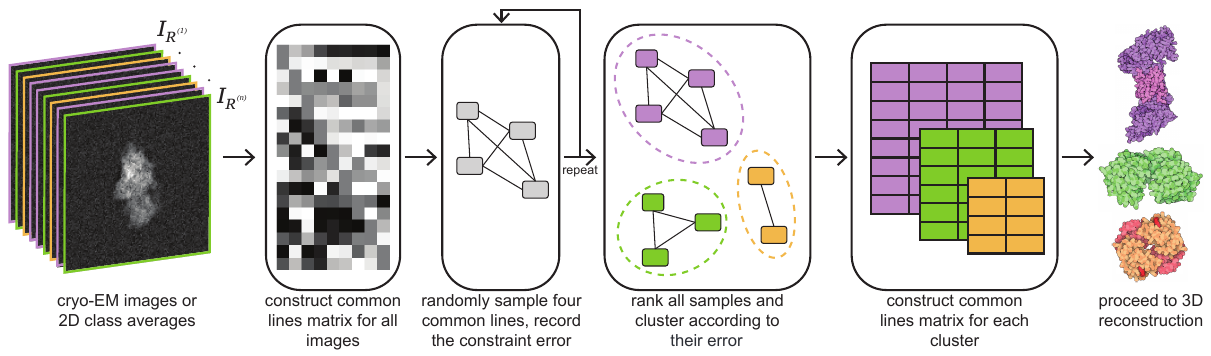}
{\caption{Algorithm for separating images of distinct molecules using algebraic constraints on common lines. The common lines matrix is first computed from an input set of images or class averages. We then apply Algorithm \ref{alg:clustering}, our clustering algorithm. After clustering, images corresponding to the same molecule can then be used for 3D reconstruction}
\label{fig:scheme}}
\end{figure}

\begin{algorithm}[!ht]
\caption{Heterogeneous clustering on common line constraints}
\label{alg:clustering}
\begin{algorithmic}[1]
\Input
{$A \in \mathbb{R}^{2n \times n}$, a common lines matrix}
\Output
$C_1, \ldots, C_r \subseteq \{1,\ldots,n\}$, a partition of all common lines into consistent clusters
\Procedure{\textsc{Clusters}}{$A$}
\State $L \gets \emptyset$ \Comment{1. Generate samples}
\For {$i \geq 1$ until sufficient} 
\State set $S \subseteq \{1,\ldots,n\}$ to be a random sample such that $|S| = 4$
\State set $A_S$ to be the submatrix of $A$ associated to the common lines in $S$
\State $A_S \gets \textsc{IRLS-ADMM}(A_S)$
\State $A_S \gets \textsc{Sinkhorn}(A_S)$
\If{converged}
\State set $M$ using \eqref{eq:M-matrix-symmetry} on $A_S$
\State set $\mathbf{v}_1,\mathbf{v}_2,\mathbf{v}_3$ using \eqref{eq:v-vectors-determinant} on $A_S$
\State $e \gets \|M - M^{\top}\|_F^2 + \|\mathbf{v}_1 - \mathbf{v}_2\|_2^2 + \|\mathbf{v}_2 - \mathbf{v}_3\|_2^2$
\State $L \gets \text{append}(L,(S,e))$
\EndIf
\EndFor
\State sort $L$ by increasing values of $e$ \Comment{2. Cluster samples}
\State $G \gets \boldsymbol{0}_{n \times n}$ \While {$L \neq \emptyset$} 
\State $(S,e) \gets \text{pop}(L)$
\State $G_{S,S} \gets \max\{G_{S,S},-\log(e)\cdot\boldsymbol{1}_{4 \times 4}\}$
\EndWhile
\State $G_{ii} \gets 0$ for all $i = 1,\ldots,n$
\State $C_1,\ldots,C_r \gets \textsc{CommunityDetection}(G)$ \Comment{Use method from~\cite{Lancichinetti_2009}}
\EndProcedure
\end{algorithmic}
\end{algorithm}

Our clustering algorithm consists of two main steps:

\medskip

\noindent \textbf{1. Generate samples:} As input, we are given a single common lines matrix $A$, from which we identify small clusters of consistent common lines. We randomly sample a set of four common lines $S$ and extract the corresponding $8 \times 4$ submatrix $A_S$ from $A$ (the choice of four common lines is explained in Section \ref{sec:clustering-heterogeneous} based on numerical experiments). We then run \textsc{IRLS-ADMM} and \textsc{Sinkhorn} on $A_S$. These methods may occasionally diverge due to numerical instability from noise in the data, as discussed in Section \ref{sec:experiments}, in which case we discard the sample. We also discard the sample if the spectral gap between the third and fourth singular value of $A_S$ is not sufficiently large (i.e., $A_S$ does not have numerical rank 3). Otherwise, we obtain a quadratic constraint satisfaction error for the sample. We record both the sample and its error, and repeat this sampling sufficiently many times.

\medskip

\noindent \textbf{2. Cluster samples:} We view the collection of samples and errors we obtain as a weighted hypergraph on $n$ vertices whose hyperedges are of size 4 corresponding to the samples and whose hyperedge weights are given by their corresponding errors.  We convert the weighted hypergraph into a weighted graph by constructing a weighted adjacency matrix whose $(i,j)$ entry is the negative logarithm of the smallest error on a hyperedge containing both common lines $i$ and $j$. We then use an unsupervised community detection algorithm on this adjacency matrix to find the clusters of consistent common lines. In our numerical experiments, we use the algorithm proposed by Lancichinetti,
Fortunato, and Kertész~\cite{Lancichinetti_2009}, which can identify overlapping communities and hierarchical structure, and depends only on a single hyperparamter controlling the scale of the hierarchies. In particular, we do not need to specify the number of clusters or their sizes.

\medskip

The clustering algorithm is detailed in Algorithm \ref{alg:clustering} and illustrated in Figure \ref{fig:scheme}.

\section{Performance on data} \label{sec:experiments}

We compare the performance of our methods to existing common lines based algorithms in the literature, namely functions in the software package ASPIRE~\cite{garrett_wright_2023_8321443} and the clustering algorithm of Verbeke et. al.~\cite{verbeke_separating_2020}. We study the problems of recovering rotations, denoising common lines, and partitioning discretely heterogeneous image sets into homogeneous subcommunities using common lines.  The tests are done on simulated data (at various levels of noise) and real data.

Our simulated data consists of image data of the 40S, 60S and 80S ribosome (available from the Electron Microscopy Data Bank~\cite{Lawson2015-rz} as entries EMD-4214, EMD-2811, and EMD-2858, respectively), generated by ASPIRE. The ribosomes and examples of their clean 2D projection images are displayed in Figure \ref{fig:simulation}. Each image is $128 \times 128$ with a pixel size of 3\r{A}. White Gaussian noise is added to each image corresponding to a specified signal-to-noise ratio (SNR). We define the SNR by taking the signal to be the average squared intensity over each pixel in the clean image and setting the noise variance to achieve the appropriate ratio. The common lines between two images are detected by finding the line projections of the two images with the highest correlation, as computed in ASPIRE.

\begin{figure}[!ht]
{\includegraphics[width=1\textwidth]{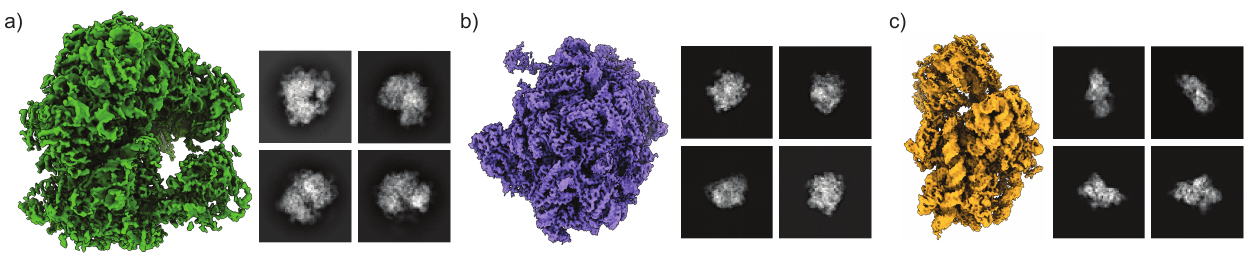}}
{\caption{3-D structures and example projection images for the three structures used for simulation. \textbf{(a)} 80S ribosome (EMD-2858) and example projection images. \textbf{(b)} 60S ribosome (EMD-2811) and example projection images. \textbf{(c)} 40S ribosome (EMD-4214) and example projection images}
\label{fig:simulation}}
\end{figure}

In our numerical experiments, we have observed that the performance of \textsc{IRLS-ADMM}, and consequently \textsc{Sinkhorn}, can depend on its initialization. In particular, we have occasionally observed divergence or vanishing of the entries of $A$ in \textsc{Sinkhorn}. This behavior appears to be due to  numerical instabilities in  \textsc{Sinkhorn} arising from noise in the data or from using too many common lines. In these cases, we can either discard such runs and restart the algorithms with new initializations, or skip using \textsc{Sinkhorn}. We address this issue in each of our tests.

\subsection{Rotation recovery}{\label{sec:rotation-recovery-tests}}

Let $R^{(1)},\ldots,R^{(n)} \in \operatorname{SO}(3)$ be the ground-truth rotations and $R_1,\ldots,R_n \in \operatorname{SO}(3)$ be the recovered rotations. Then Theorem \ref{thm:unique-rotations} states that there exists a unique rotation $Q \in \operatorname{SO}(3)$ such that $R^{(i)} = R_iQ$ for all $1 \leq i \leq n$. In other words, $Q$ can be found by solving the orthogonal Procrustes problem
\begin{equation}{\label{eq:squared-procrustes-test}}
    \min_{Q \in \operatorname{SO}(3)} \frac{1}{n} \left\|\begin{pmatrix}
        R^{(1)} \\
        \vdots \\
        R^{(n)}
    \end{pmatrix} - \begin{pmatrix}
        R_1 \\
        \vdots \\
        R_n
    \end{pmatrix}Q\right\|_F^2
\end{equation}
The solution to this problem can be found using SVD \cite{Schönemann1966}. We note that for $n = 1$, there is a simple relation between the Procrustes error and the angular error between two rotation matrices.  For $R,S \in SO(3)$, it holds 
$ \|R - S\|_F^2 = \langle R - S, R - S\rangle = \|R^\top R\|_F^2 - 2\langle R, S\rangle + \|S^\top S\|_F^2 
    = \|I_{3 \times 3}\|_F^2 - 2\langle R, S\rangle + \|I_{3 \times 3}\|_F^2 = 2\cdot 3 - 2\cdot \text{tr}(R^\top S)$.
Hence the angular error between $R$ and $S$ is
\begin{align*}{\label{eq:average-angular-err}}
    \theta &= \text{arccos}\Big{(}\frac{\text{tr}(R^\top S) - 1}{2}\Big{)} = \text{arccos}\Big{(}\frac{3 - \frac{1}{2}\|R - S\|_F^2 - 1}{2}\Big{)}.
\end{align*}

We compare our method to the procedure in ASPIRE on simulated data. Given common lines, the rotation recovery algorithm used in ASPIRE is the synchronization with voting procedure as described in \cite{shkolniskySinger1}. We use 30 images of the 60S, 80S, and 40S ribosomes at $\text{SNR} = 0.125,0.25,0.5,1$. For each macromolecule and SNR, we generate 50 sets of 30 random ground-truth rotations and their corresponding noisy images. We report the average Procrustes error \eqref{eq:squared-procrustes-test} per image (i.e., the Procrustes error divided by the number of images).

When running this test, we chose to only use our \textsc{IRLS-ADMM} algorithm followed by \textsc{Rotations}, and omitted the \textsc{Sinkhorn} row and column scaling step due to observed numerical instabilities of \textsc{Sinkhorn} with noise in the data or too many common lines.
Also, Remark \ref{rem:row-signs} states that our algorithm is only guaranteed to recover the ground-truth pure common lines matrix up to a global scale, and the sign of this scale produces two sets of rotations that differ by left-multiplication with $J = \text{diag}(-1,-1,1)$. This sign flip corresponds to the chiral ambiguity in cryo-EM, as explained in Remark \ref{rem:careful}. Thus in our tests we report the best rotational error amongst the two possible sets of rotations.

The results for rotation recovery are displayed in Table \ref{table:rotation-test}. Notably, at lower SNR our method is consistently more accurate than ASPIRE.

\subsection{Denoising common lines}{\label{sec:denoising-tests}}

The problem we consider here is the following: given a \textit{noisy} common lines matrix, how well can we recover the ground-truth clean pure common lines matrix?  If $A,B \in \mathbb{R}^{2n \times n}$ are the ground-truth and recovered pure common lines matrix respectively, then we  measure this error to be
\begin{equation}{\label{eq:denoising-err}}
    \min_{\lambda \in \mathbb{R}} \frac{1}{n}\left\|A - \lambda B\right\|_F^2
\end{equation}
since there is a global scale ambiguity in the recovered pure common lines matrix as discussed in Remark \ref{rem:row-signs}. The above problem is a least-squares problem in $\lambda$ and hence has a closed-form solution.

We compare our methods to ASPIRE on simulated data as follows. We run the rotation recovery algorithm of ASPIRE based on common lines to obtain rotations $R^{(1)},\ldots,R^{(n)} \in SO(3)$. Then we construct a recovered pure common lines matrix $B$ from these rotations by using the factorization \eqref{eq:A-factorization}. We then compare the denoising error \eqref{eq:denoising-err} from the ground-truth clean common lines matrix $A$ to the output of our \textsc{IRLS-ADMM} method and to the matrix $B$. As before, the simulated data consists of 30 images of the 60S, 80S, and 40S ribosomes at $\text{SNR} = 0.125,0.25,0.5,1$, and we report the average denoising error \eqref{eq:denoising-err} per image over 50 runs. We assume we have the ground-truth chiral information when recovering rotations with ASPIRE.

The results for common line denoising are displayed in Table \ref{table:denoising-table}.  Again our \textsc{IRLS-ADMM} method outperforms ASPIRE for denoising common lines matrices, particularly at low SNR.

\begin{table}[h]
\centering
\renewcommand{\arraystretch}{1.25}
\begin{tabular}{|c|c|c|c|}
\hhline{|-|-|-|-|}
\multirow{3}{*}{macromolecule} & \multirow{3}{*}{SNR} & \multicolumn{2}{c|}{average Procrustes rotation error \eqref{eq:squared-procrustes-test}} \\
\hhline{|~|~|-|-|}
& & \makecell[c]{\Cref{alg:irls-admm} \& \\\Cref{alg:rot-recovery}} & \multicolumn{1}{c|}{ASPIRE} \\
\hhline{:=:=:=:=:}
\multirow{4}{*}{\shortstack[c]{EMD-2811 \\ (60S ribosome)}} & 0.125 & \textbf{3.6840} & 4.3426 \\
\hhline{|~|-|-|-|}
& 0.25 & \textbf{3.3215} & 3.6042 \\
\hhline{|~|-|-|-|}
& 0.5 & \textbf{1.5512} & 2.2704 \\
\hhline{|~|-|-|-|}
& 1 & \textbf{0.4215} & 0.4877 \\
\hhline{:=:=:=:=:}
\multirow{4}{*}{\shortstack[c]{EMD-2858 \\ (80S ribosome)}} & 0.125 & \textbf{1.7903} & 2.5943 \\
\hhline{|~|-|-|-|}
& 0.25 & \textbf{0.8038} & 1.3175 \\
\hhline{|~|-|-|-|}
& 0.5 & \textbf{0.1576} & 0.2543 \\
\hhline{|~|-|-|-|}
& 1 & 0.0259 & \textbf{0.0240} \\
\hhline{:=:=:=:=:}
\multirow{4}{*}{\shortstack[c]{EMD-4214 \\ (40S ribosome)}} & 0.125 & \textbf{4.0377} & 4.5281 \\
\hhline{|~|-|-|-|}
& 0.25 & \textbf{3.9079} & 4.0104 \\
\hhline{|~|-|-|-|}
& 0.5 & 3.6375 & \textbf{3.3994} \\
\hhline{|~|-|-|-|}
& 1 & 2.8271 & \textbf{2.1334} \\
\hhline{|-|-|-|-|}
\end{tabular}
\medskip
\caption{The average rotation recovery error from 30 simulated images  of macromolecules \newline at various SNR, 50 runs each. Bold values indicate the algorithm with lower error.}
\label{table:rotation-test}
\end{table}

\begin{table}[h]
\centering
\renewcommand{\arraystretch}{1.25}
\begin{tabular}{|c|c|c|c|}
\hhline{|-|-|-|-|}
\multirow{2}{*}{macromolecule} & \multirow{2}{*}{SNR} & \multicolumn{2}{c|}{average denoising error \eqref{eq:denoising-err}} \\
\hhline{|~|~|-|-|}
& & \makecell[c]{\Cref{alg:irls-admm}} & \multicolumn{1}{c|}{ASPIRE} \\
\hhline{:=:=:=:=:}
\multirow{4}{*}{\shortstack[c]{EMD-2811 \\ (60S ribosome)}} & 0.125 & \textbf{14.1982} & 18.2274 \\
\hhline{|~|-|-|-|}
& 0.25 & \textbf{12.7786} & 17.8805 \\
\hhline{|~|-|-|-|}
& 0.5 & \textbf{8.2349} & 16.8199 \\
\hhline{|~|-|-|-|}
& 1 & \textbf{4.1908} & 6.1365 \\
\hhline{:=:=:=:=:}
\multirow{4}{*}{\shortstack[c]{EMD-2858 \\ (80S ribosome)}} & 0.125 & \textbf{11.2948} & 17.3292 \\
\hhline{|~|-|-|-|}
& 0.25 & \textbf{6.8850} & 13.3778 \\
\hhline{|~|-|-|-|}
& 0.5 & \textbf{2.3152} & 3.6690 \\
\hhline{|~|-|-|-|}
& 1 & 0.5680 & \textbf{0.3512} \\
\hhline{:=:=:=:=:}
\multirow{4}{*}{\shortstack[c]{EMD-4214 \\ (40S ribosome)}} & 0.125 & \textbf{15.5635} & 18.0612 \\
\hhline{|~|-|-|-|}
& 0.25 & \textbf{15.1112} & 18.0862 \\
\hhline{|~|-|-|-|}
& 0.5 & \textbf{14.0420} & 18.0069 \\
\hhline{|~|-|-|-|}
& 1 & \textbf{11.6964} & 15.7237 \\
\hhline{|-|-|-|-|}
\end{tabular}
\medskip
\caption{The average denoising error of recovered pure common lines matrices \newline from 30 simulated images of macromolecules at various SNR, 50 runs each. \newline Bold values indicate the algorithm with lower error.}
\label{table:denoising-table}
\end{table}

\subsection{Clustering heterogeneous image sets}{\label{sec:clustering-heterogeneous}}

We test the performance of our algorithm \textsc{Clusters} for clustering (see Section \ref{sec:heterogeneity-application}) on simulated and real data.

The success of clustering is measured by the adjusted Rand index \cite{hubert1985} (ARI) between the ground-truth clusters and the recovered clusters. The range of this index is $-\infty < \text{ARI} \leq 1$, with $\text{ARI} = 1$ if the two partitions are identical. The ARI is a corrected-for-chance version of the Rand index, meaning that it is the expected Rand index for the cluster and is equal to 0 if every element is placed in a random cluster.

While any number of common lines $|S|$ can be sampled on line 4 of \textsc{Clusters}, we found that sampling four common lines at a time was effective for a number of reasons: $|S| = 4$ enforces a non-trivial rank 3 constraint, and the small number of common lines allowed us to both generate many samples rapidly and improve the numerical stability of the \textsc{Sinkhorn} scaling procedure. In addition, the \textsc{CommunityDetection} algorithm we use is the one described in \cite{Lancichinetti_2009}.

\medskip

\subsubsection{Simulated data}

We generate a dataset containing three clusters, with $n = 5+30+15 = 50$ images from the 40S, 60S, and 80S ribosome respectively, from which we construct a common lines matrix.

\textsc{Clusters} achieved perfect clustering ($\text{ARI} = 1$) at $\text{SNR} = 10$, $\text{ARI} = 0.8581$ at $\text{SNR} = 5$, and $\text{ARI} = 0.3286$ at $\text{SNR} = 1$. The clusters found at $\text{SNR} = 5$ are displayed in Figure \ref{fig:clustering_simulated}, which shows that only one pair of images were placed in incorrect clusters.

\begin{figure}[!ht]
{\includegraphics[width=1\textwidth]{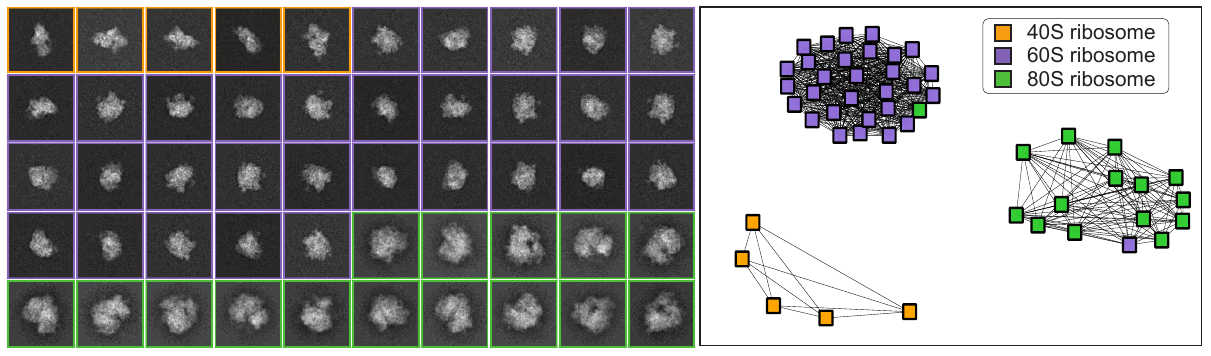}}
{\caption{Clustering results for $n=50$ simulated images with SNR = 5. Images are size $128 \times 128$ with pixel size of 3\r{A}, and are colored according to the ground truth labels. Using \textsc{Clusters} achieves \newline ARI = 0.8581 and only one pair of images are incorrectly clustered}
\label{fig:clustering_simulated}}
\end{figure}

\medskip

\subsubsection{Real data}
Our real data consists of a subset of 2D class averages computed from the experimental data described in Verbeke et. al.~\cite{verbeke_separating_2020}. The subset we consider consists of two clusters with $n = 47 + 28 = 75$ images corresponding to the 60S and 80S ribosomes respectively. Each class average is $96 \times 96$ with a pixel size of 4.4\r{A}. We use the labels from ~\cite{verbeke_separating_2020} as ground-truth for clustering.

Figure \ref{fig:clustering_real} shows the clusters found by our algorithm \textsc{Clusters}, achieving ARI = 0.8440 and misclassifying only three images.

The clustering algorithm used in \cite{verbeke_separating_2020} is based on performing community detection on a nearest-neighbours graph constructed using \textit{Euclidean distances} between the best-matching line projections between every pair of images. We stress that our clustering algorithm uses a completely distinct aspect of common lines data: the \textit{positions} of the common lines.  As a proof of concept for our constraints, we do not make use of the correlations between the common lines at all, unlike \cite{verbeke_separating_2020}.  
The test for \textsc{Clusters} only uses a subset of the dataset in \cite{verbeke_separating_2020}, which has $n = 100$ images and includes images with unknown labels. If we compare only the 60S and 80S images that were clustered, then \textsc{Clusters} achieves a similar performance, where one additional image is misclassified by \textsc{Clusters} compared to \cite{verbeke_separating_2020}.

\begin{figure}[!t]
\centering
{\includegraphics[width=0.85\textwidth]{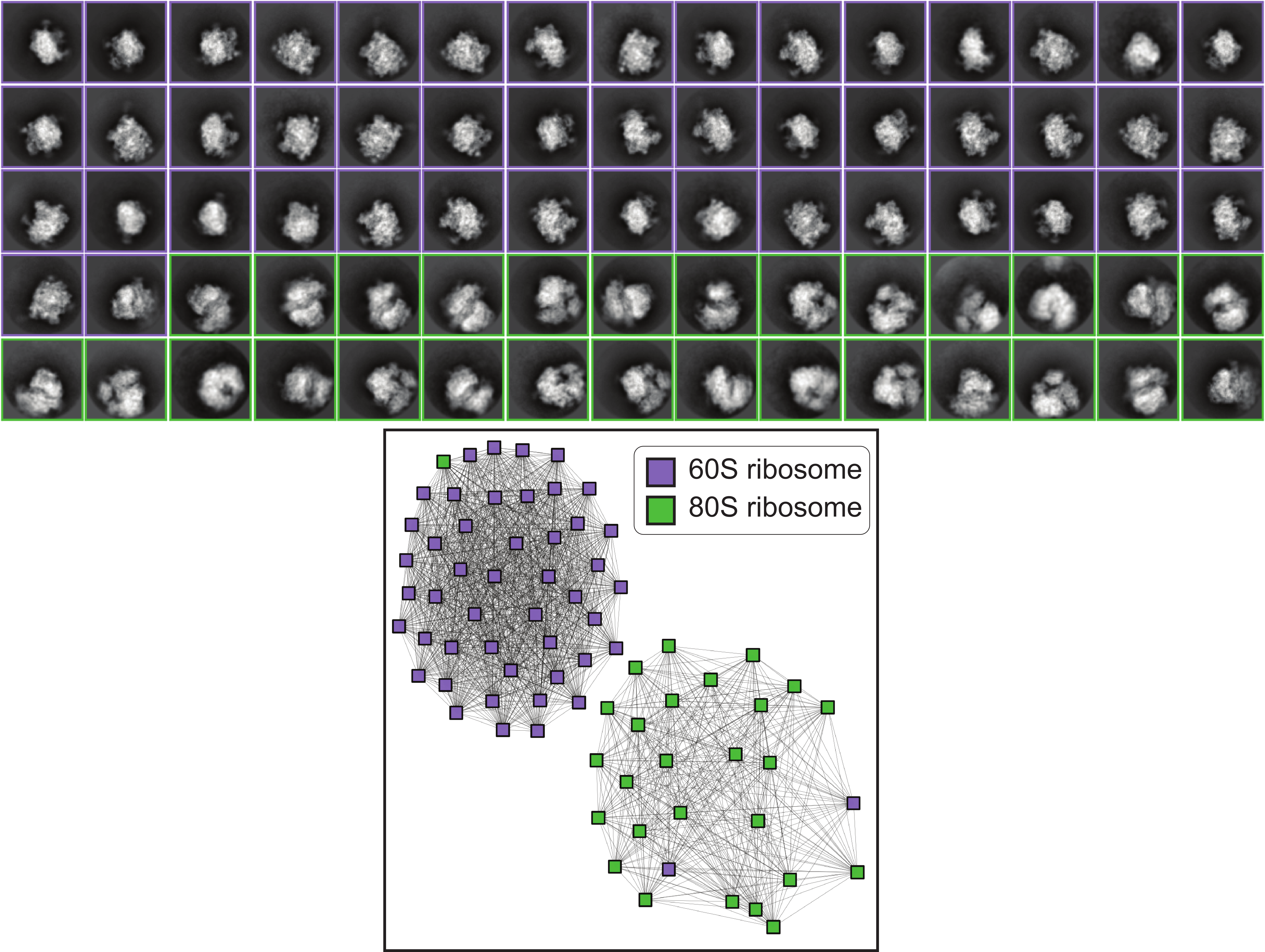}}
{\caption{Clustering results for $n=75$ 2D class averages from EMPIAR-10268 computed as described in ~\cite{verbeke_separating_2020}. Images are size $96 \times 96$ with a pixel size of 4.4\r{A}, and are colored according to the ground truth labels. Using \textsc{Clusters} achieves ARI = 0.8440 and only three images are incorrectly clustered}
\label{fig:clustering_real}}
\end{figure}

\section{Conclusion}{\label{sec:conclusion}}

This paper revisited the fundamental topic of common lines in cryo-EM image processing.
We discussed a novel approach for dealing with common lines, based on a certain $2n \times n$ matrix encoding the common lines between $n$ projection images.
We proved that if the $2 \times 1$ blocks of the matrix are properly scaled, then the matrix satisfies nice algebraic constraints: a low-rank condition and several sparse quadratic constraints.  
The new formulation operates directly on common lines data, and is fully global in that it does not require angular reconstitution or voting procedures at all.
It opens the door to different and potentially
 more robust approaches to computational tasks involving common lines.
Using the algebraic constraints, we adapted optimization algorithms from other domains to give new methods to denoise common lines data, and recover the 3D rotations underlying noisy images. Numerical experiments show that these methods have increased accuracy at low SNR, compared to existing methods based on common lines.
We also explored a setting where traditional common lines methods fail to apply -- cryo-EM datasets with discrete heterogeneity -- by proposing a sampling-based process to cluster the images homogeneous subcommunities based on our algebraic constraints.
Experiments with simulated and real data show the method performs well when applied to images with high noise.

Although there is clear promise, several future directions could be pursued for further improvements.  
Firstly, in \Cref{sec:algo} the optimization algorithm building on \cite{scalingAmit} is quite complex. Matrix scaling problems as in \cite{scalingAmit} and our work are an interesting variation on the problem of matrix completion; would other optimization approaches perform better?
Secondly, extensions to molecules with nontrivial point group symmetries would be useful (and currently are a focus in other common lines research). Perhaps our formulation can suggest another way to incorporate symmetries into common lines methods.
Lastly, in the application to discrete heterogeneity, we neglected correlation scores between the common lines, on which \cite{verbeke_separating_2020} relied.  It is likely better to use both the scores and the algebraic constraint errors. 

\medskip
\medskip

\begin{Backmatter}

\noindent \textbf{Funding Statement.} T.M. is supported by a Mathematical Institute Scholarship at the University of Oxford.
A.D. is supported in part by NSF DMS 1937215. 
E.V. is supported by the Simons Foundation Math+X Investigator award to Amit Singer.
J.K. is supported in part by NSF DMS 2309782, NSF CISE-IIS 2312746, and start-up grants from the College of Natural
Science and Oden Institute at the University of Texas at Austin.

\medskip
\medskip

\noindent\textbf{Competing Interests.} The authors declare no competing interests exist.

\medskip
\medskip

\noindent \textbf{Data Availability Statement.} Data and replication code are available at
\begin{equation}{\label{link:github}} 
\textup{\href{https://github.com/ozitommi/algebraic-common-lines}{https://github.com/ozitommi/algebraic-common-lines}}   
\end{equation}

\medskip
\medskip

\noindent \textbf{Author Contributions.} All authors conceived of the project, designed the algorithms and developed the mathematical theory.  T.M., A.D. and E.V. prepared the data. 
T.M. and A.D. wrote the software and performed the numerical experiments.  All authors wrote the manuscript and approved of its final submission.

\bibliography{main}
\bibliographystyle{achemso}

\end{Backmatter}

\clearpage
\markboth{APPENDIX}{APPENDIX}
\appendix
\numberwithin{equation}{section}

\section{Additional proofs}{\label{appendix:a}}

\begin{proof}[Proof of Proposition \ref{prop:length}]{\label{proof:prop-length}}
    Since orthogonal matrices preserve norms, we have
    \begin{align}{\label{eq:rji}}
        1 &= \|\mathbf{r}_3^{(i)}\|_2^2 = \left\|R^{(j)} \mathbf{r}_3^{(i)}\right\|_2^2 = \left\|\begin{pmatrix}
            1 & & \\
            & -1 & \\
            & & 1
        \end{pmatrix}\begin{pmatrix}
            & & 1 \\
            & 1 & \\
            1 & & 
        \end{pmatrix}\begin{pmatrix}
            {\mathbf{r}_1^{(j)}}^\top \\
            {\mathbf{r}_2^{(j)}}^\top \\
            {\mathbf{r}_3^{(j)}}^\top 
        \end{pmatrix} \mathbf{r}_3^{(i)}\right\|_2^2 = \left\|\begin{pmatrix}
            {\mathbf{r}_3^{(j)}}^\top \\
            -{\mathbf{r}_2^{(j)}}^\top \\
            {\mathbf{r}_1^{(j)}}^\top 
        \end{pmatrix} \mathbf{r}_3^{(i)}\right\|_2^2 \nonumber \\
        &= \langle \mathbf{r}_3^{(j)},  \mathbf{r}_3^{(i)} \rangle^2 + \left\|\begin{pmatrix}
            -\langle \mathbf{r}_2^{(j)}, \mathbf{r}_3^{(i)} \rangle \\
            \langle \mathbf{r}_1^{(j)}, \mathbf{r}_3^{(i)} \rangle
        \end{pmatrix} \right\|_2^2 = \langle \mathbf{r}_3^{(i)},  \mathbf{r}_3^{(j)} \rangle^2 + \left\|\mathbf{a}_{ji}\right\|_2^2
    \end{align}
    Similarly, we have
    \begin{equation}{\label{eq:rij}}
        1 = \langle \mathbf{r}_3^{(i)},  \mathbf{r}_3^{(j)} \rangle^2 + \left\|\mathbf{a}_{ij}\right\|_2^2
    \end{equation}
    Thus, equating \eqref{eq:rji} with \eqref{eq:rij} and simplifying, we obtain $\|\mathbf{a}_{ij}\|_2^2 = \|\mathbf{a}_{ji}\|_2^2$.
\end{proof}

\begin{proof}[Proof of Proposition \ref{prop:determinant}]
Let $D = \mathrm{det}\begin{pmatrix}
    \mathbf{r}_3^{(i)} & \mathbf{r}_3^{(j)} & \mathbf{r}_3^{(k)}
\end{pmatrix}$. Then for each $i$ we have
\begin{align}{\label{eq:ijik}}
    D &= \mathrm{det}\left(\begin{pmatrix}
        1 & & \\
        & -1 & \\
        & & 1 
    \end{pmatrix}\begin{pmatrix}
            & & 1 \\
            & 1 & \\
            1 & & 
        \end{pmatrix}\begin{pmatrix}
            {\mathbf{r}_1^{(i)}}^\top \\
            {\mathbf{r}_2^{(i)}}^\top \\
            {\mathbf{r}_3^{(i)}}^\top 
        \end{pmatrix}\begin{pmatrix}
    \mathbf{r}_3^{(i)} & \mathbf{r}_3^{(j)} & \mathbf{r}_3^{(k)}
\end{pmatrix}\right) = \mathrm{det}\left(\begin{pmatrix}
            {\mathbf{r}_3^{(i)}}^\top \\
            -{\mathbf{r}_2^{(i)}}^\top \\
            {\mathbf{r}_1^{(i)}}^\top 
        \end{pmatrix}\begin{pmatrix}
    \mathbf{r}_3^{(i)} & \mathbf{r}_3^{(j)} & \mathbf{r}_3^{(k)}
\end{pmatrix}\right) \nonumber \\
&= \mathrm{det}\begin{pmatrix}
    1 & \langle \mathbf{r}_3^{(i)},\mathbf{r}_3^{(j)} \rangle & \langle \mathbf{r}_3^{(i)},\mathbf{r}_3^{(k)} \rangle \\
    0 & -\langle \mathbf{r}_2^{(i)}, \mathbf{r}_3^{(j)} \rangle & -\langle \mathbf{r}_2^{(i)}, \mathbf{r}_3^{(k)} \rangle  \\
    0 & \langle \mathbf{r}_1^{(i)}, \mathbf{r}_3^{(j)} \rangle & \langle \mathbf{r}_1^{(i)}, \mathbf{r}_3^{(k)} \rangle
\end{pmatrix} = \mathrm{det}\begin{pmatrix}
    -\langle \mathbf{r}_2^{(i)}, \mathbf{r}_3^{(j)} \rangle & -\langle \mathbf{r}_2^{(i)}, \mathbf{r}_3^{(k)} \rangle \\
    \langle \mathbf{r}_1^{(i)}, \mathbf{r}_3^{(j)} \rangle & \langle \mathbf{r}_1^{(i)}, \mathbf{r}_3^{(k)} \rangle
\end{pmatrix} = \mathrm{det}\begin{pmatrix}
    \mathbf{a}_{ij} & \mathbf{a}_{ik},
\end{pmatrix}
\end{align}
by expanding the determinant along the first column. Similarly for each $j$ and $k$ we have
\begin{equation}{\label{eq:jijk}}
    D = \mathrm{det}\begin{pmatrix}
    \langle \mathbf{r}_3^{(j)},\mathbf{r}_3^{(i)} \rangle & 1 & \langle \mathbf{r}_3^{(j)},\mathbf{r}_3^{(k)} \rangle \\
    -\langle \mathbf{r}_2^{(j)}, \mathbf{r}_3^{(i)} \rangle & 0 & -\langle \mathbf{r}_2^{(j)}, \mathbf{r}_3^{(k)} \rangle  \\
    \langle \mathbf{r}_1^{(j)}, \mathbf{r}_3^{(i)} \rangle & 0 & \langle \mathbf{r}_1^{(j)}, \mathbf{r}_3^{(k)} \rangle
\end{pmatrix} = -\mathrm{det}\begin{pmatrix}
    -\langle \mathbf{r}_2^{(j)}, \mathbf{r}_3^{(i)} \rangle & -\langle \mathbf{r}_2^{(j)}, \mathbf{r}_3^{(k)} \rangle \\
    \langle \mathbf{r}_1^{(j)}, \mathbf{r}_3^{(i)} \rangle & \langle \mathbf{r}_1^{(j)}, \mathbf{r}_3^{(k)} \rangle
\end{pmatrix} = -\mathrm{det}\begin{pmatrix}
    \mathbf{a}_{ji} & \mathbf{a}_{jk}
\end{pmatrix}
\end{equation}
\begin{equation}{\label{eq:kikj}}
    D = \mathrm{det}\begin{pmatrix}
    \langle \mathbf{r}_3^{(k)},\mathbf{r}_3^{(i)} \rangle & \langle \mathbf{r}_3^{(k)},\mathbf{r}_3^{(j)} \rangle & 1 \\
    -\langle \mathbf{r}_2^{(k)}, \mathbf{r}_3^{(i)} \rangle & -\langle \mathbf{r}_2^{(k)}, \mathbf{r}_3^{(j)} \rangle & 0  \\
    \langle \mathbf{r}_1^{(k)}, \mathbf{r}_3^{(i)} \rangle & \langle \mathbf{r}_1^{(k)}, \mathbf{r}_3^{(j)} \rangle & 0
\end{pmatrix} = \mathrm{det}\begin{pmatrix}
    -\langle \mathbf{r}_2^{(k)}, \mathbf{r}_3^{(i)} \rangle & -\langle \mathbf{r}_2^{(k)}, \mathbf{r}_3^{(j)} \rangle \\
    \langle \mathbf{r}_1^{(k)}, \mathbf{r}_3^{(i)} \rangle & \langle \mathbf{r}_1^{(k)}, \mathbf{r}_3^{(j)} \rangle
\end{pmatrix} = \mathrm{det}\begin{pmatrix}
    \mathbf{a}_{ki} & \mathbf{a}_{kj}
\end{pmatrix}
\end{equation}
Thus, equating~\eqref{eq:ijik},~\eqref{eq:jijk}, and~\eqref{eq:kikj}, we obtain $\mathrm{det}\begin{pmatrix}
    \mathbf{a}_{ij} & \mathbf{a}_{ik}
\end{pmatrix} = -\mathrm{det} \begin{pmatrix}
    \mathbf{a}_{ji} & \mathbf{a}_{jk}
\end{pmatrix} = \mathrm{det}\begin{pmatrix}
    \mathbf{a}_{ki} & \mathbf{a}_{kj}
\end{pmatrix}$.
\end{proof}

\begin{proof}[Proof of Theorem \ref{thm:unique-rotations}]
We first prove the theorem in the case $n = 3$. Given $R_1,R_2,R_3 \in SO(3)$, we have $A = \psi(R_1,R_2,R_3) = \psi(R_1Q,R_2Q,R_3Q)$ for any $Q \in SO(3)$ since
\begin{align*}
     \begin{pmatrix} - {\mathbf{r}_2^{(1)}}^{\!\top} \\ {\mathbf{r}_1^{(1)}}^{\!\top} \\ 
\vdots \\  -{\mathbf{r}_2^{(3)}}^{\!\top} \\ {\mathbf{r}_1^{(3)}}^{\!\top}  \end{pmatrix}_{6 \times 3} \!\!\! \begin{pmatrix}  \mathbf{r}_3^{(1)} & \mathbf{r}_3^{(2)} & \mathbf{r}_3^{(3)} \end{pmatrix}_{3 \times n} = \begin{pmatrix} \begin{pmatrix}
        - {\mathbf{r}_2^{(1)}}^{\!\top} \\ {\mathbf{r}_1^{(1)}}^{\!\top}
    \end{pmatrix}Q \\ 
\vdots \\  \begin{pmatrix}
        - {\mathbf{r}_2^{(3)}}^{\!\top} \\ {\mathbf{r}_1^{(3)}}^{\!\top}
    \end{pmatrix}Q \\ \end{pmatrix}_{6 \times 3} \!\!\! \begin{pmatrix}  Q^\top \mathbf{r}_3^{(1)} & Q^\top \mathbf{r}_3^{(2)} & Q^\top \mathbf{r}_3^{(3)} \end{pmatrix}_{3 \times 3}
\end{align*}
Hence, the fibres of $\psi$ contain at least a copy of $SO(3)$. By fixing $Q = R_1^\top$, we may assume that $R_1 = I$, so that $A = \psi(R_1,R_2,R_3) = \psi(I,R_2R_1^\top,R_3R_1^\top)$. To prove the statement, we therefore need to show that fibers of the map $\psi$ with $R_1 = I$ fixed,
\begin{equation}
\begin{split}
\overline{\psi}: \operatorname{SO}(3) \times \operatorname{SO}(3) &\longrightarrow \mathbb{R}^{6 \times 3} \\
    (R^{(1)},R^{(2)}) &\mapsto \psi(I,R^{(1)},R^{(2)}),
\end{split}
\end{equation}
generically consist of only one point. Our strategy to show this is to set up a corresponding system of polynomial equations for a random instance of two rotations $R^{(1)}$ and $R^{(2)}$, and then solve the system numerically, showing that it has only one solution. 
It is sufficient to solve the polynomial system over the complex numbers $\mathbb{C}$, and exhibit that there is a unique solution even over $\mathbb{C}$.
We parameterize $R_1, R_2 \in \operatorname{SO}(3)$ using the Euler-Rodriguez formula:
\begin{equation*}
    R_i = \begin{pmatrix} a_i^2+b_i^2-c_i^2-d_i^2 & 2(b_ic_i-a_id_i) & 2(b_id_i + a_ic_i) \\
                                       2(b_ic_i+a_id_i) & a_i^2+c_i^2-b_i^2-d_i^2 & 2(c_id_i - a_ib_i) \\
                                       2(b_id_i-a_ic_i) & 2(c_id_i+a_ib_i) & a_i^2+d_i^2-b_i^2-c_i^2 \end{pmatrix}
\end{equation*}
where $a_i,b_i,c_i,d_i \in \mathbb{R}$ such that $a_i^2 + b_i^2 + c_i^2 + d_i^2 = 1$. We construct one pure common lines matrix $A = \psi(I,R_1,R_2)$ whose entries are polynomial in $a_i,b_i,c_i,d_i$, and another pure common lines matrix $\overline{A} = \psi(I,\overline{R_1},\overline{R_2})$, from random rotation matrices $\overline{R_1},\overline{R_2} \in \operatorname{SO}(3)$, whose entries are real. Setting $A = \overline{A}$ gives us a system of polynomial equations, which we solve over $\mathbb{C}$ using the package HomotopyContinuation.jl in Julia \cite{homotopyContinuation}. Every rotation under the Euler-Rodriguez parametrization has two parametrizations, namely $(a,b,c,d)$ and $(-a,-b,-c,-d)$ representing the same matrix in $\operatorname{SO}(3)$. 
By parameterizing $R_1$ and $R_2$, we therefore expect to find 4 real solutions to the system and therefore one point in the fiber, which is indeed what we find using homotopy continuation. The computation is carried out in the file {fiber.jl} in our Github repository \eqref{link:github}.

Now with the theorem true for $n = 3$, it in fact follows that the theorem is true for all $n > 3$: if $A_1 = \psi(R_1,\ldots,R_n)$ and $A_2 = \psi(S_1,\ldots,S_n)$, where $A_1 = A_2$ is a generic point in the image of $\psi$, then we need to show that there exists a unique $Q \in \operatorname{SO}(3)$ such that $S_i = R_iQ$ for all $i = 1,\ldots,n$. Choosing two $6 \times 3$ pure common lines submatrices of $A_1 = A_2$ corresponding to indices $\{i,j,k\}$ and $\{i,\ell,m\}$, the theorem in the case $n = 3$ implies that there exist $Q_{ijk},Q_{i\ell m} \in \operatorname{SO}(3)$ such that $(S_i,S_j,S_k) = (R_iQ_{ijk},R_jQ_{ijk},R_jQ_{ijk})$ and $(S_i,S_j,S_\ell) = (R_iQ_{i\ell m},R_\ell Q_{i\ell m},R_m Q_{i\ell m})$. In particular, $R_iQ_{ijk} = R_iQ_{i\ell ,}$, so $Q_{ijk} = Q_{i\ell m}$. Thus, any triplet of rotations are related by the same matrix $Q$.
\end{proof}

\begin{proof}[Proof of Proposition \ref{prop:suffice-locally}]
Let $\mathcal{V}_n \subseteq \mathbb{R}^{2n \times n}$ denote the cone over the common lines variety (i.e., $\mathcal{V}_n$ is the the smallest algebraic variety containing all scalar multiples of all pure common lines matrices).  
Let $\mathcal{W}_n \subseteq \mathbb{R}^{2n \times n}$ denote the variety defined by the \!\! polynomial constraints in \Cref{prop:suffice-locally}.  
Note that since the constraints are invariant to global scaling, we have $\mathcal{V}_n \subseteq \mathcal{W}_n$.
By \cite{bochnak2013real}, in order to show that $\mathcal{V}_n$ is an irreducible component of $\mathcal{W}_n$ we need to show that
\begin{equation}\label{eq:want1}
\dim(\mathcal{V}_n) = \dim(\mathcal{W}_n, A)
\end{equation}
holds a generic point $A \in \mathcal{V}_n$, where $\dim(\cdot)$ and $\dim(\cdot, \cdot)$ respectively denote the dimension of a variety and the local dimension of a variety at a point. This is analogous to computing the dimension of a manifold by computing the dimension of its tangent space at a point.

Recall $\mathcal{V}_n$ is the cone over the common lines variety, which is the smallest algebraic variety containing the image of the map $\psi$ in \eqref{phi-map}.  
Thus using the fiber dimension theorem (the algebraic geometric analog of the rank-nullity theorem from linear algebra), \Cref{thm:unique-rotations} implies 
\begin{equation}\label{eq:know1}
\dim(\mathcal{V}_n) = 3n-2.
\end{equation}
Next we  compute $\dim(\mathcal{W}_n, A)$. We first note that $\mathcal{W}_n \subseteq \mathcal{X}_n$, where $\mathcal{X}_n$ is the algebraic variety of all rank $\leq 3$ matrices in $\mathbb{R}^{2n \times n}$ defined by the vanishing of $4 \times 4$ minors.
Note the map $\rho : \mathbb{R}^{2n \times 3} \times \mathbb{R}^{n \times 3} \rightarrow \mathcal{V}_n$ given by $\rho(B,C) = B C^{\top}$ parameterizes $\mathcal{V}_n$ with $9$-dimensional fibers over each  rank-$3$ point \cite{levin2023finding}. This holds because for each $3 \times 3$ invertible matrix $M$, $\rho(B,C) = \rho(BM^{-1},CM^\top)$.  Hence 
\begin{equation} \label{eq:for-M2}
\dim(\mathcal{W}_n, A) + 9 = \dim(\rho^{-1}(\mathcal{W}_n), (B,C))
\end{equation}
for $(B,C) \in \mathbb{R}^{2n \times 3} \times \mathbb{R}^{n \times n}$ such that $\rho(B,C) = A$.  
We now compute the right-hand side of \eqref{eq:for-M2} in the computer algebra system Macaulay2 \cite{M2}. Specifically, we differentiate the defining constraints of $\mathcal{W}_n$ with respect to $B$ and $C$, noting that the rank-3 constraint is enforced automatically by $\rho$ and the other constraints in \Cref{prop:suffice-locally} are biquadratic and bilinear in $(B,C)$.  We evaluate the Jacobian matrix at a point $(B,C)$ in the fiber of $\rho$ over $A$, where $A \in \mathcal{V}_n$ is generated as a random scaling of $\psi(R^{(1)}, \ldots, R^{(n)})$ where $R^{(i)}$ are random.
Generically, the nullity of the Jacobian matrix equals $\dim(\rho^{-1}(\mathcal{W}_n), (B,C))$.  
Performing numerical computations for $n=3, \ldots, 50$ in double-precision arithmetic yields numerical nullities of $3n+7$. 
Comparing \eqref{eq:for-M2} with \eqref{eq:know1} implies \eqref{eq:want1} as desired.  
\end{proof}

\begin{proof}[Proof of Theorem \ref{thm:row-column-scales}]
Since nonzero row and column scaling preserves rank, we may assume without loss of generality that $\lambda_{i1} = \lambda_{1j} = 1$ for all $i,j = 2,\ldots,n$: this is achieved by choosing $\mu_1 = \tau_1 = 1$, $\mu_i = \lambda_{i1}^{-1}$, and $\tau_j = \lambda_{1j}^{-1}$. We will show that $\lambda_{ij} = \lambda_{k\ell}$ for all $i,j,k,\ell = 2,\ldots,n$, $i \neq j$, $k \neq \ell$. {In other words, we will prove that the matrix $\Lambda$ must have the form in \eqref{eq:Lambda-matrix}.}

First we prove that $\lambda_{ij} = \lambda_{ik}$ for all $i,j,k = 2,\ldots,n$, $i \neq j$, $i \neq k$. There are three cases to consider: $i < j < k$, $j < i < k$, and $j < k < i$. Suppose we are in the first case. Since $\text{rank}(B) = 3$, every $4 \times 4$ minor of $B$ vanishes. Thus, choosing row indices 1, 2, $2i-1$, and $2i$, and column indices 1, $i$, $j$, and $k$, we have
\begin{align}{\label{eq:expand-det}}
    \text{det}\begin{pmatrix}
        \mathbf{0} & \mathbf{a}_{1i} & \mathbf{a}_{1j} & \mathbf{a}_{1k} \\
        \mathbf{a}_{i1} & \mathbf{0} & \lambda_{ij}\mathbf{a}_{ij} & \lambda_{ik}\mathbf{a}_{ik}
    \end{pmatrix} = \lambda_{ik}\text{det}\begin{pmatrix}
        \mathbf{a}_{1i} & \mathbf{a}_{1j}
    \end{pmatrix}\text{det}\begin{pmatrix}
        \mathbf{a}_{i1} & \mathbf{a}_{ik}
    \end{pmatrix} - \lambda_{ij}\text{det}\begin{pmatrix}
        \mathbf{a}_{1i} & \mathbf{a}_{1k}
    \end{pmatrix}\text{det}\begin{pmatrix}
        \mathbf{a}_{i1} & \mathbf{a}_{ij}
    \end{pmatrix} = 0
\end{align}
But since $A$ is a pure common lines matrix, $\text{det}\begin{pmatrix}
        \mathbf{a}_{1i} & \mathbf{a}_{1j}
    \end{pmatrix} = -\text{det}\begin{pmatrix}
        \mathbf{a}_{i1} & \mathbf{a}_{ij}
    \end{pmatrix}$ and \text{det}$\begin{pmatrix}
        \mathbf{a}_{i1} & \mathbf{a}_{ik}
    \end{pmatrix} = -\text{det}\begin{pmatrix}
        \mathbf{a}_{1i} & \mathbf{a}_{1k}
    \end{pmatrix}$ by Proposition \ref{prop:determinant}, all of which are nonzero since $A$ is generic. Thus $\lambda_{ij} = \lambda_{ik}$.

Expanding the $4 \times 4$ minor with the same choice of row and column indices for the remaining two cases gives us
\begin{align*}
    \text{det}\begin{pmatrix}
        \mathbf{0} & \mathbf{a}_{1j} & \mathbf{a}_{1i} & \mathbf{a}_{1k} \\
        \mathbf{a}_{i1} & \lambda_{ij}\mathbf{a}_{ij} & \mathbf{0} & \lambda_{ik}\mathbf{a}_{ik}
    \end{pmatrix} = -\text{det}\begin{pmatrix}
        \mathbf{0} & \mathbf{a}_{1i} & \mathbf{a}_{1j} & \mathbf{a}_{1k} \\
        \mathbf{a}_{i1} & \mathbf{0} & \lambda_{ij}\mathbf{a}_{ij} & \lambda_{ik}\mathbf{a}_{ik}
    \end{pmatrix} = 0
\end{align*}
and
\begin{align*}
    \text{det}\begin{pmatrix}
        \mathbf{0} & \mathbf{a}_{1j} & \mathbf{a}_{1k} & \mathbf{a}_{1i} \\
        \mathbf{a}_{i1} & \lambda_{ij}\mathbf{a}_{ij}  & \lambda_{ik}\mathbf{a}_{ik} & \mathbf{0}
    \end{pmatrix} = \text{det}\begin{pmatrix}
        \mathbf{0} & \mathbf{a}_{1i} & \mathbf{a}_{1j} & \mathbf{a}_{1k} \\
        \mathbf{a}_{i1} & \mathbf{0} & \lambda_{ij}\mathbf{a}_{ij} & \lambda_{ik}\mathbf{a}_{ik}
    \end{pmatrix} = 0
\end{align*}
respectively, which brings us back to \eqref{eq:expand-det}.

Suppose that $3 \leq i \leq n-1$. The $4 \times 4$ minor whose row indices are $2i - 1$, $2i$, $2(i+1) - 1$, and $2(i+1)$, and whose column indices are 1, $i-1$, $i$, and $i+1$ is
\begin{align*}
    \text{det}\begin{pmatrix}
        \mathbf{a}_{i1} & \lambda_{i,i-1}\mathbf{a}_{i,i-1} & \mathbf{0} & \lambda_{i,i+1}\mathbf{a}_{i,i+1} \\
        \mathbf{a}_{i+1,1} & \lambda_{i+1,i-1}\mathbf{a}_{i+1,i-1} & \lambda_{i+1,i}\mathbf{a}_{i+1,i} & \mathbf{0}
    \end{pmatrix} = \text{det}\begin{pmatrix}
        \mathbf{0} & \lambda_{i,i+1}\mathbf{a}_{i,i+1} & \mathbf{a}_{i1} & \lambda_{i,i-1}\mathbf{a}_{i,i-1} \\
        \lambda_{i+1,i}\mathbf{a}_{i+1,i} & \mathbf{0} & \mathbf{a}_{i+1,1} & \lambda_{i+1,i-1}\mathbf{a}_{i+1,i-1} 
    \end{pmatrix}
\end{align*}
We showed earlier that $\lambda^\prime := \lambda_{i,i+1} = \lambda_{i,i-1}$ and $\lambda^{\prime\prime} := \lambda_{i+1,i} = \lambda_{i+1,i-1}$. Thus expanding, we obtain
\begin{equation}{\label{eq:expand-det-2}}
    \begin{aligned}
    \text{det}\begin{pmatrix}
        \mathbf{0} & \lambda^\prime\mathbf{a}_{i,i+1} & \mathbf{a}_{i1} & \lambda^\prime\mathbf{a}_{i,i-1} \\
        \lambda^{\prime\prime}\mathbf{a}_{i+1,i} & \mathbf{0} & \mathbf{a}_{i+1,1} & \lambda^{\prime\prime}\mathbf{a}_{i+1,i-1} 
    \end{pmatrix} &= \lambda^\prime{\lambda^{\prime\prime}}^2\text{det}\begin{pmatrix}
        \mathbf{a}_{i1} & \mathbf{a}_{i,i+1}
    \end{pmatrix}\text{det}\begin{pmatrix}
        \mathbf{a}_{i+1,i-1} & \mathbf{a}_{i+1,i}
    \end{pmatrix} \\
    &- {\lambda^\prime}^2\lambda^{\prime\prime}\text{det}\begin{pmatrix}
        \mathbf{a}_{i,i-1} & \mathbf{a}_{i,i+1}
    \end{pmatrix}\text{det}\begin{pmatrix}
        \mathbf{a}_{i+1,1} & \mathbf{a}_{i+1,i}
    \end{pmatrix} = 0
    \end{aligned}
\end{equation}
After dividing both sides of the last equation by $\lambda^\prime \lambda^{\prime\prime}$, we obtain $\lambda^\prime = \lambda^{\prime\prime}$. Lastly, the $4 \times 4$ minor with the same choice of row and column indices when $i = 2$ is
\begin{align*}
    \text{det}\begin{pmatrix}
        \mathbf{a}_{21} & \mathbf{0} & \lambda_{23}\mathbf{a}_{23} & \lambda_{24}\mathbf{a}_{24} \\
        \mathbf{a}_{31} & \lambda_{32}\mathbf{a}_{32} & \mathbf{0} & \lambda_{34}\mathbf{a}_{34} 
    \end{pmatrix} = \text{det}\begin{pmatrix}
        \mathbf{0} & \lambda^\prime\mathbf{a}_{23} & \mathbf{a}_{21} & \lambda^\prime\mathbf{a}_{24} \\
        \lambda^{\prime\prime}\mathbf{a}_{32} & \mathbf{0} & \mathbf{a}_{31} & \lambda^{\prime\prime}\mathbf{a}_{34} 
    \end{pmatrix} = 0
\end{align*}
since $\lambda^\prime := \lambda_{23} = \lambda_{24}$ and $\lambda^{\prime\prime} := \lambda_{32} = \lambda_{34}$, which brings us back to \eqref{eq:expand-det-2}. Thus we have proven that $\lambda_{ij} = \lambda_{k\ell}$ for all $i,j,k,\ell = 2,\ldots,n$, $i \neq j$, $k \neq \ell$. Let $\lambda$ be their common value. Then
\begin{equation}{\label{eq:Lambda-matrix}}
    \Lambda = \begin{pmatrix}
0 & 1 & 1 & \ldots & 1 & 1  \\
1 & 0 & \lambda & \ldots & \lambda & \lambda \\
1 & \lambda & 0 & \ldots & \lambda & \lambda \\
\vdots & \vdots &  \vdots & \ddots & \vdots & \vdots \\
1 & \lambda & \lambda & \ldots & 0 & \lambda  \\
1 & \lambda & \lambda & \ldots & \lambda & 0
\end{pmatrix}
\end{equation}
and we may obtain \eqref{eq:row-column-scale} by choosing {$\mu_1 = \frac{1}{\lambda}$, $\mu_2 = \ldots = \mu_n = 1$ and $\tau_1 = 1$, $\tau_2 = \ldots = \tau_n = \lambda$.}

Now we prove the last statement of the theorem. What we have shown so far is that $\text{rank}(B) = 3$ implies that there exists scales $\mu_i$ and $\tau_j$ such that $\lambda_{ij} = \mu_i\tau_j$. {Since $\lambda_{ij} = \lambda_{ji}$, we have
$$\mu_i\tau_j = \mu_j\tau_i$$
for all for all $i,j = 1,\ldots,n$. If $\boldsymbol{\mu} = (\mu_i)_{i=1}^n$ and $\boldsymbol{\tau} = (\tau_j)_{j = 1}^n$, then this means that the matrix $\boldsymbol{\mu}\boldsymbol{\tau}^\top \in \mathbb{R}^{n \times n}$ is symmetric. This is true if and only if $\boldsymbol{\mu} = \boldsymbol{\tau}$. In particular,
\begin{equation}{\label{eq:mu-sq=tau-sq}}
    \mu_i = \tau_i
\end{equation}
If $B$ furthermore satisfies the determinant constraints, then we have
\begin{equation*}
    \mu_i^2\tau_j\tau_k\text{det}\begin{pmatrix}
        \mathbf{a}_{ij} & \mathbf{a}_{ik}
    \end{pmatrix} = -\mu_j^2\tau_i\tau_k\text{det}\begin{pmatrix}
        \mathbf{a}_{ji} & \mathbf{a}_{jk}
    \end{pmatrix} = \mu_k^2\tau_i\tau_j\text{det}\begin{pmatrix}
        \mathbf{a}_{ki} & \mathbf{a}_{kj}
    \end{pmatrix}
\end{equation*}
for all $1 \leq i < j < k \leq n$, which implies that
\begin{equation*}
    \mu_i^2\tau_j\tau_k = \mu_j^2\tau_i\tau_k = \mu_k^2\tau_i\tau_j
\end{equation*}
Substituting \eqref{eq:mu-sq=tau-sq}, we obtain
\begin{equation*}
    \tau_i^2\tau_j\tau_k = \tau_j^2\tau_i\tau_k = \tau_k^2\tau_i\tau_j \implies \tau_i = \tau_j = \tau_k
\end{equation*}
after dividing by $\tau_i\tau_j\tau_k$. Denoting the common value of the above equation on the right by $\tau$, we find that $\mu_i = \tau$. Hence, $\lambda_{ij} = \tau$ for all $i,j = 1,\ldots,n$, $i \neq j$, which gives the desired result.}
\end{proof}

\section{Least squares problems}{\label{appendix:b}}

In Section \ref{sec:sinkhorn}, we consider the least squares problems \eqref{sinkhorn-d1} and \eqref{sinkhorn-d2}, which are
\begin{equation*}
    \boldsymbol{\mu} = \argmin_{\|\boldsymbol{\mu}\|_2=1} \|\text{diag}(\boldsymbol{\mu})M - (\text{diag}(\boldsymbol{\mu})M)^\top\|_F^2
\end{equation*}
\begin{equation*}
    \boldsymbol{\tau} = \argmin_{\|\boldsymbol{\tau}\|_2=1} \|M\text{diag}(\boldsymbol{\tau}) - (M\text{diag}(\boldsymbol{\tau}))^\top\|_F^2
\end{equation*}
These have solutions
\begin{equation*}
    \min_{\|\boldsymbol{\mu}\|_2=1} \left\|N_L \cdot \boldsymbol{\mu}\right\|_2^2
\end{equation*}
\begin{equation*}
    \min_{\|\boldsymbol{\tau}\|_2=1} \left\|N_R \cdot \boldsymbol{\tau}\right\|_2^2
\end{equation*}
respectively, where $N_L \in \mathbb{R}^{n \times n}$ with
\begin{equation}{\label{eq:Nl-solution}}
    (N_L)_{ij} = \begin{cases}
    \displaystyle \sum_{\substack{k = 1,...,n \\ k \neq i}} M_{ik}^2 &\quad \text{ if } i = j \\
        -M_{ij}M_{ji} &\quad \text{ if } i \neq j
    \end{cases}
\end{equation}
and $N_R \in \mathbb{R}^{n \times n}$ with
\begin{equation}{\label{eq:Nr-solution}}
    (N_R)_{ij} = \begin{cases}
    \displaystyle \sum_{\substack{k = 1,...,n \\ k \neq i}} M_{ki}^2 &\quad \text{ if } i = j \\
        -M_{ij}M_{ji} &\quad \text{ if } i \neq j
    \end{cases}
\end{equation}

We also consider the least squares problems \eqref{determinant-scales-LS-1} and \eqref{determinant-scales-LS-2}, which are
\begin{align}{\label{determinant-scales-LS-1-appendix}}
 \min_{\|\boldsymbol{\mu}\|_2=1} \|(\boldsymbol{\mu} \;\triangle\; \mathbf{v}_1) - (\boldsymbol{\mu} \;\triangle\; \mathbf{v}_2)\|_2^2 + \|(\boldsymbol{\mu} 
 \;\triangle\; \mathbf{v}_2) - (\boldsymbol{\mu} \;\triangle\; \mathbf{v}_3)\|_2^2
\end{align}
\begin{align}{\label{determinant-scales-LS-2-appendix}}
 \min_{\|\boldsymbol{\tau}\|_2=1} \quad \|(\boldsymbol{\tau} \;\triangle_1\; \mathbf{v}_1) - (\boldsymbol{\tau} \;\triangle_1\; \mathbf{v}_2)\|_2^2 + \|(\boldsymbol{\tau} \;\triangle_2\; \mathbf{v}_2) - (\boldsymbol{\tau} \;\triangle_2\; \mathbf{v}_3)\|_2^2
\end{align}
where
\begin{equation}{\label{eq:mu-triangle}}
\begin{aligned}
\boldsymbol{\mu}\;\triangle\;\mathbf{v}_1 &:= \begin{pmatrix}
\mu_i\text{sgn}(\text{det}\begin{pmatrix}
        \mathbf{a}_{ij} & \mathbf{a}_{ik}
    \end{pmatrix})\sqrt{|\text{det}\begin{pmatrix}
        \mathbf{a}_{ij} & \mathbf{a}_{ik}
    \end{pmatrix}|}
\end{pmatrix}_{1 \leq i < j < k \leq n}\\
\vspace{1em}\\
\boldsymbol{\mu}\;\triangle\;\mathbf{v}_2 &:= \begin{pmatrix}
\mu_j\text{sgn}(\text{det}\begin{pmatrix}
        \mathbf{a}_{ji} & \mathbf{a}_{jk}
    \end{pmatrix})\sqrt{|\text{det}\begin{pmatrix}
        \mathbf{a}_{ji} & \mathbf{a}_{jk}
    \end{pmatrix}|}   
\end{pmatrix}_{1 \leq i < j < k \leq n}\\
\vspace{1em}\\
\hspace{-2em}\boldsymbol{\mu}\;\triangle\;\mathbf{v}_3 &:= \begin{pmatrix}
\mu_k\text{sgn}(\text{det}\begin{pmatrix}
        \mathbf{a}_{ki} & \mathbf{a}_{kj}
    \end{pmatrix})\sqrt{|\text{det}\begin{pmatrix}
        \mathbf{a}_{ki} & \mathbf{a}_{kj}
    \end{pmatrix}|}  
\end{pmatrix}_{1 \leq i < j < k \leq n}
\end{aligned}
\end{equation}
and 
\begin{equation}{\label{eq:tau-triangle}}
\begin{aligned}
\boldsymbol{\tau} \;\triangle_1\; \mathbf{v}_1 &:= \begin{pmatrix}\tau_j\text{det}\begin{pmatrix}
        \mathbf{a}_{ij} & \mathbf{a}_{ik}
    \end{pmatrix}
\end{pmatrix}_{1 \leq i < j < k \leq n}
&
\boldsymbol{\tau} \;\triangle_1\; \mathbf{v}_2 &:= \begin{pmatrix}
-\tau_i\text{det}\begin{pmatrix}
        \mathbf{a}_{ji} & \mathbf{a}_{jk}
    \end{pmatrix}
\end{pmatrix}_{1 \leq i < j < k \leq n} \\ \vspace{5em} & \\
\boldsymbol{\tau} \;\triangle_2\; \mathbf{v}_2 &:= \begin{pmatrix}
-\tau_k\text{det}\begin{pmatrix}
        \mathbf{a}_{ji} & \mathbf{a}_{jk}
    \end{pmatrix}
\end{pmatrix}_{1 \leq i < j < k \leq n}
&
\boldsymbol{\tau} \;\triangle_2\; \mathbf{v}_3 &:= \;\;\;\begin{pmatrix}
\tau_j\text{det}\begin{pmatrix}
        \mathbf{a}_{ki} & \mathbf{a}_{kj}
    \end{pmatrix}
\end{pmatrix}_{1 \leq i < j < k \leq n}
\end{aligned}
\end{equation}
are all vectors of length ${n \choose 3}$. The solution to problem \eqref{determinant-scales-LS-1-appendix} is 
\begin{equation*}
    \min_{\|\boldsymbol{\mu}\|_2=1} \left\|(D_{L,1} + D_{L,2}) \cdot \boldsymbol{\mu}\right\|_2^2
\end{equation*}
where $D_{L,1}, D_{L,2} \in \mathbb{R}^{n \times n}$ with
\begin{equation}{\label{eq:DL1-solution}}
    (D_{L,1})_{pq} = \left\{\begin{aligned}
    \displaystyle & &\sum_{\substack{1 \leq j < k \leq n \\ j \neq p, k > p}} & \text{sgn}(v_{pj,pk})|v_{pj,pk}| &\quad \text{ if } p = q \\
        \displaystyle &- &\sum_{\max\{p,q\} < k \leq n} &  \text{sgn}(v_{pq,pk}v_{qp,qk})\sqrt{|v_{pq,pk}|}\sqrt{|v_{qp,qk}|} &\quad \text{ if } p \neq q
    \end{aligned}\right.
\end{equation}
\begin{equation}{\label{eq:DL2-solution}}
    (D_{L,2})_{pq} = \left\{\begin{aligned}
    \displaystyle & &\sum_{\substack{1 \leq j < k \leq n \\ j \neq p, k > p}} & \text{sgn}(v_{pk,pj})|v_{pj,pk}| &\quad \text{ if } p = q \\
        \displaystyle &- &\sum_{\max\{p,q\} < k \leq n} &  \text{sgn}(v_{pq,pk}v_{qp,qk})\sqrt{|v_{pq,pk}|}\sqrt{|v_{qp,qk}|} &\quad \text{ if } p \neq q
    \end{aligned}\right.
\end{equation}

\medskip

\noindent and the solution to problem \eqref{determinant-scales-LS-2-appendix} is
\begin{equation*}
    \min_{\|\boldsymbol{\tau}\|_2=1} \left\|(D_{R,1} + D_{R,2}) \cdot \boldsymbol{\tau}\right\|_2^2
\end{equation*}
where $D_{R,1}, D_{R,2} \in \mathbb{R}^{n \times n}$ with
\begin{equation}{\label{eq:DR1-solution}}
    (D_{R,1})_{pq} = \left\{\begin{aligned}
    \displaystyle & &\sum_{\substack{1 \leq j < k \leq n \\ j \neq p, k > p}} & v_{jp,jk}^2 &\quad \text{ if } p = q \\
        \displaystyle &- &\sum_{\max\{p,q\} < k \leq n} &  v_{pq,pk}v_{qp,qk} &\quad \text{ if } p \neq q
    \end{aligned}\right.
\end{equation}
\begin{equation}{\label{eq:DR2-solution}}
    (D_{R,2})_{pq} = \left\{\begin{aligned}
    \displaystyle & &\sum_{\substack{1 \leq k < j \leq n \\ j \neq p, k < p}} & v_{jk,jp}^2 &\quad \text{ if } p = q \\
        \displaystyle &- &\sum_{1 \leq k < \max\{p,q\}} &  v_{pk,pq}v_{qk,qp} &\quad \text{ if } p \neq q
    \end{aligned}\right.
\end{equation}

\noindent where $v_{ij,ik} := \text{det}\begin{pmatrix}
        \mathbf{a}_{ij} & \mathbf{a}_{ik}
    \end{pmatrix}$.
    
\end{document}